\documentclass[11pt,oneside]{article}

\usepackage[utf8]{inputenc} \usepackage[T1]{fontenc}    \usepackage[english]{babel}

\usepackage{scrextend}

\usepackage{amsfonts}
\usepackage{nicefrac}       \usepackage{microtype}      \usepackage{lmodern}
\usepackage{amssymb,amsmath,amsthm}

\usepackage{stmaryrd}
\SetSymbolFont{stmry}{bold}{U}{stmry}{m}{n}

\usepackage{bbm,bm}
\usepackage{latexsym}
\usepackage{xcolor}

\usepackage{enumerate}
\usepackage{verbatim}
\usepackage{booktabs}       

\usepackage{url}            \usepackage[colorlinks=true]{hyperref}
\usepackage[numbers,sort&compress,square,comma]{natbib}
\usepackage[capitalise,noabbrev]{cleveref}

\usepackage{parskip}
\usepackage{geometry}
\usepackage[short]{optidef}
\usepackage{algorithmic,algorithm}

\usepackage{thmtools}

\declaretheorem{theorem}

\declaretheorem{lemma}

\declaretheorem{proposition}

\declaretheoremstyle[qed=$\square$]{definitionwithend}
\declaretheorem[style=definitionwithend]{definition}
\declaretheorem[style=definitionwithend]{assumption}
\declaretheorem[style=definitionwithend]{example}
\declaretheorem[style=definitionwithend]{remark}

\crefname{assumption}{Assumption}{Assumptions}
\crefname{conjecture}{Conjecture}{Conjectures}

\newcommand{\ceil}[1]{\ensuremath{\left\lceil #1 \right\rceil}}
\newcommand{\abs}[1]{\ensuremath{\left\lvert #1 \right\rvert}}

\newcommand{\by}{\times}
 
\newcommand{\norm}[1]{\ensuremath{\left\lVert #1 \right\rVert}}
\newcommand{\ip}[1]{\ensuremath{\left\langle #1 \right\rangle}}
\newcommand{\grad}{\ensuremath{\nabla}}

\newcommand{\set}[1]{\left\{#1\right\}}

\newcommand{\mb}{\mathbf}

\def\N{{\mathbb{N}}}

\def\R{{\mathbb{R}}}
\def\S{{\mathbb{S}}}

\def\cA{{\cal A}}
\def\cB{{\cal B}}

\def\cK{{\cal K}}

\def\cU{{\cal U}}

\def\cX{{\cal X}}

\DeclareMathOperator*{\argmin}{arg\,min}

\DeclareMathOperator{\Diag}{Diag}
\DeclareMathOperator{\diag}{diag}
\DeclareMathOperator{\tr}{tr}

\newenvironment{smallpmatrix}
    {\left(
    \begin{smallmatrix}} 
    {\end{smallmatrix}
    \right)
    }

\newcommand{\framedheader}[3]{
  \framebox[\textwidth]{
    \vbox{
      \vspace{2mm}
      \hbox to \textwidth {\hspace{1em}\today \hfill #1\hspace{1em}}
      \vspace{4mm}
      \hbox to \textwidth {\hfill \Large{#2} \hfill}
      \vspace{2mm}
    }
  }
  \vspace*{4mm}
}
 \usepackage{letltxmacro}
\LetLtxMacro\orgvdots\vdots
\LetLtxMacro\orgddots\ddots

\makeatletter
\DeclareRobustCommand\vdots{\mathpalette\@vdots{}}
\newcommand*{\@vdots}[2]{\sbox0{$#1\cdotp\cdotp\cdotp\m@th$}\sbox2{$#1.\m@th$}\vbox{\dimen@=\wd0 \advance\dimen@ -3\ht2 \kern.5\dimen@
\dimen@=\wd2 \advance\dimen@ -\ht2 \dimen2=\wd0 \advance\dimen2 -\dimen@
    \vbox to \dimen2{\offinterlineskip
      \copy2 \vfill\copy2 \vfill\copy2 }}}
\DeclareRobustCommand\ddots{\mathinner{\mathpalette\@ddots{}\mkern\thinmuskip
  }}
\newcommand*{\@ddots}[2]{\sbox0{$#1\cdotp\cdotp\cdotp\m@th$}\sbox2{$#1.\m@th$}\vbox{\dimen@=\wd0 \advance\dimen@ -3\ht2 \kern.5\dimen@
\dimen@=\wd2 \advance\dimen@ -\ht2 \dimen2=\wd0 \advance\dimen2 -\dimen@
    \vbox to \dimen2{\offinterlineskip
      \hbox{$#1\mathpunct{.}\m@th$}\vfill
      \hbox{$#1\mathpunct{\kern\wd2}\mathpunct{.}\m@th$}\vfill
      \hbox{$#1\mathpunct{\kern\wd2}\mathpunct{\kern\wd2}\mathpunct{.}\m@th$}}}}

\DeclareRobustCommand\bddots{\mathinner{\mathpalette\@bddots{}\mkern\thinmuskip
  }}
\newcommand*{\@bddots}[2]{\sbox0{$#1\cdotp\cdotp\cdotp\m@th$}\sbox2{$#1.\m@th$}\vbox{\dimen@=\wd0 \advance\dimen@ -3\ht2 \kern.5\dimen@
\dimen@=\wd2 \advance\dimen@ -\ht2 \dimen2=\wd0 \advance\dimen2 -\dimen@
    \vbox to \dimen2{\offinterlineskip
      \hbox{$#1\mathpunct{\kern\wd2}\mathpunct{\kern\wd2}\mathpunct{.}\m@th$}\vfill
      \hbox{$#1\mathpunct{\kern\wd2}\mathpunct{.}\m@th$}\vfill
      \hbox{$#1\mathpunct{.}\m@th$}}}}
\makeatother \usepackage{soul}

 \DeclareMathOperator{\supp}{supp}
\DeclareMathOperator{\sign}{sign}
\DeclareMathOperator{\dist}{dist}
\DeclareMathOperator{\rip}{RIP}
\newcommand{\ripup}{\rip^+}
\newcommand{\ripdown}{\rip^-}
\newcommand{\journalonly}[1]{}
\usepackage{siunitx}
\usepackage{subcaption}

\usepackage{graphicx}
\graphicspath{{figures/}}

\makeatletter
\renewcommand\paragraph{\@startsection{paragraph}
    {4}
    {\z@}
    {3.25ex \@plus1ex \@minus.2ex}
    {-1em}
    {\normalfont\normalsize\bfseries\maybe@addperiod}}
\newcommand{\maybe@addperiod}[1]{#1\@addpunct{.}}
\makeatother

\begin{document}

\title{Sharpness and well-conditioning of nonsmooth convex formulations in statistical signal recovery}
\author{
    Lijun Ding
    \thanks{Department of Mathematics, University of California San Diego. San Diego, CA, USA(\texttt{l2ding@ucsd.edu})}
    \and
    Alex L.\ Wang
    \thanks{Purdue University. West Lafayette, IN, USA(\texttt{wang5984@purdue.edu})}
}
\date{\today}
\maketitle

\begin{abstract}
    We study a sample complexity vs.\ conditioning tradeoff in modern signal recovery problems (including sparse recovery, low-rank matrix sensing, covariance estimation, and abstract phase retrieval), where convex optimization problems are built from sampled observations.
We begin by introducing  a set of condition numbers related to sharpness in $\ell_p$ or Schatten-$p$ norms ($p\in[1,2]$) of a nonsmooth formulation for these problems.
Then, we show that these condition numbers become dimension \emph{independent} constants 
in each of the example signal recovery problems
once 
the sample size exceeds some constant multiple of the recovery threshold.
Structurally, this result ensures that 
the inaccuracy in the recovered signal due to both observation noise and optimization error is well-controlled. Algorithmically, such a result ensures that a new restarted mirror descent method achieves nearly-dimension-independent linear convergence to the signal.
This new first-order method is general and applies to any sharp convex function in an $\ell_p$ or Schatten-$p$ norm ($p\in[1,2]$).
 \end{abstract}

\section{Introduction}
\label{sec:intro}
This paper studies a \textit{sample complexity vs.\ convex optimization conditioning} tradeoff in modern signal recovery problems such as 
sparse recovery  \cite{candes2005decoding}, low-rank matrix  sensing \cite{recht2010guaranteed,cai2015rop}, phase retrieval \cite{candes2013phaselift}, and covariance estimation \cite{chen2015exact}. 
To illustrate the setup, the challenges, and our results concretely, we will only consider the abstract phase retrieval problem (henceforth phase retrieval) in the introduction.
The general setup and results follow a similar pattern and are deferred to later sections.

Phase retrieval asks us  to recover an unknown rank-one positive semidefinite (PSD) matrix $X^\natural\in\S^n$ given a vector of $m$ measurements $b = \cA(X^\natural)$ where the linear map $\cA:\S^n\to\R^m$ is defined by $\cA(X)_i = g_i^\intercal X g_i$ for i.i.d.\ $g_i\sim N(0, I_n/m)$. 
Here and throughout, $\S^n$ and $\S^n_+$ denote the vector space of symmetric matrices and the cone of PSD matrices respectively. 
Concretely, we are asked to
\begin{center}
	recover $X^\natural \in \S^n$ from $(b,\cA)$ where $b = \cA(X^\natural)\in \R^m$.
\end{center}

In~\cite{candes2014solving}, Cand\`es and Li 
established an \emph{exact recovery} result: If $m=\Theta\left(n\right)$, then, with high probability (w.h.p.)\footnote{An event happens w.h.p. if the probability is larger than $1-\exp(-cm)$ for some numerical constant $c$.}, $X^\natural$ is the unique optimizer for the convex optimization problem
\begin{align}
	\label{eq:phret_intro}
	\min_{X\in\S^n}\set{\tr(X):\, \begin{array}
			{l}
			\cA(X) = b\\
			X\in\S^n_+
	\end{array}}.
\end{align}
Here, the choice of the objective function $\tr(X)$ and the constraint $X\in\S^n_+$ reflect prior knowledge that the true signal is low rank and PSD respectively.

\paragraph{Issues and challenges} Unfortunately, convexity and uniqueness of solution are not the end of the story as 
\eqref{eq:phret_intro} may not be well-conditioned. In such a situation, the recovery process may suffer from:
\begin{itemize}
	\item \textit{measurement error}: if the observation $b$ is corrupted by noise, the optimal solution of \eqref{eq:phret_intro} may deviate greatly from $X^\natural$ or may even cease to 
	exist;
	\item \textit{optimization error}: the problem \eqref{eq:phret_intro} may admit near-optimal and near-feasible solutions in terms of function value and constraint violation that are far from the true signal $X^\natural$;
	\item \textit{slow algorithm convergence}:
    the iteration complexity of existing first-order algorithms for solving \eqref{eq:phret_intro} may depend polynomially on the dimension due to polynomially poor conditioning of 
\eqref{eq:phret_intro}.\footnote{These polynomial factors may be even worse when measuring error in terms of distance to the true solution as opposed to function value.}
	Such dependence weakens the applicability of existing first-order methods on large-scale instances of \eqref{eq:phret_intro}.
\end{itemize}

\paragraph{Our contribution} 
In this paper, we introduce a notion of conditioning related to \emph{sharpness} for \eqref{eq:phret_intro} based on an \emph{unconstrained nonsmooth 
	reformulation}
of \eqref{eq:phret_intro}
that allows us to quantitatively control each of the
above issues.
To measure the conditioning of \eqref{eq:phret_intro}, we introduce the following penalized nonsmooth version of \eqref{eq:phret_intro}:
\begin{align}
	\label{eq:phret_penalized_intro}
	\min_{X\in\S^n} F_{r,\ell}(X) \coloneqq \tr(X) + r \norm{\cA(X) - b}_1 + \ell \dist_{1}(X,\S^n_+).
\end{align}
Here, the distance $\dist_{1}(X,\S^n_+) \coloneqq \inf_{Y\in \S^n_+} \norm{Y-X}_1$ is measured in the Schatten-1 norm $\norm{\cdot}_1$, i.e., the sum of the singular values, and
$r,\ell\geq 0$ are penalty parameters on the violations of $\cA(X) = b$ and $X\in\S^n_+$ respectively.
To clean up the exposition, we will understand ``$\ell_p$ norm'' of a matrix to mean its Schatten-$p$ norm and ``$\ell_p$ norm'' of a vector mean its usual $\ell_p$ norm.

We emphasize that we penalize the violation of $\cA(X) = b$ by $\norm{\cA(X) - b}_1$ and not the more common least-squares error $\norm{\cA(X) - b}_2^2$ 
and we use $\ell_1$ norm in the distance function.
The term $\norm{\cA(X) - b}_1$ is sometimes referred to as the least absolute deviation error term~\cite{portnoy1997gaussian} and is used as a robust version of the least-squares error.
Although this choice leads to a nonsmooth problem, it also
opens the possibility for the function \eqref{eq:phret_penalized_intro} to be $\mu$-sharp around $X^\natural$ w.r.t.\ $\ell_1$ for some $\mu>0$. That is, for all $X\in\S^n$
\begin{equation}\label{eq: sharp_phase-retrieval}
	\underbrace{\tr(X) + r \norm{\cA(X) - b}_1 + \ell \dist_{\ell_1}(X,\S^n_+)}_{F_{r,\ell}(X)} - 
	\underbrace{\tr(X^\natural)}_{F_{r,\ell}(X^\natural)} \geq \mu \norm{X-X^\natural}_1.
\end{equation}
When such a bound holds, we will think of the parameters $(\mu,r,\ell)$ as partially describing the conditioning of \eqref{eq:phret_intro}.
Finally, in order to make our notion of conditioning ``aware'' of scaling, we additionally 
track the Lipschitz constant, $L$, of \eqref{eq:phret_penalized_intro} with respect to $\ell_1$: 
\begin{equation}\label{eq: Lipschitz_phase-retrieval}
	|F_{r,\ell}(X)-F_{r,\ell}(Y)|\leq L \norm{X-Y}_1, \quad \text{for all}\; X,Y \in \S^n.
\end{equation}
We use the set of numbers $(\mu,r,\ell,L)$ to measure the conditioning of \eqref{eq:phret_intro}.
The sharpness parameter $\mu$ in some sense will control absolute notions of error, while the condition number $\kappa \coloneqq L/\mu$ will control relative notions of error and convergence rates of first-order methods.\footnote{One may replace the $\ell_1$ norm with other norms in the right hand side of these definitions \eqref{eq: sharp_phase-retrieval} and \eqref{eq: Lipschitz_phase-retrieval}. The benefit of the $\ell_1$ norm, as we will discuss, is that the conditioning in our applications can be shown to be \emph{dimension-independent}.}   

We now describe our main results in the setting of phase retrieval:
\begin{itemize}
	\item \emph{Well-conditioning}: By adapting well-known proofs for exact recovery \cite{candes2005decoding,recht2010guaranteed}, we show in  \cref{thm: sampleVSconditioning} that \eqref{eq:phret_penalized_intro} is $\mu$-sharp 
	and $L$-Lipschitz with 
	$\mu,L = \Theta(1)$ for $r,\ell =\Theta(1)$ w.h.p.\ once $m = \Theta(n)$.
Almost identical results on dimension-independent conditioning for the other signal recovery problems are also presented in \Cref{thm: sampleVSconditioning}.
\item \emph{Optimization error}: In practice, we are unable to solve \eqref{eq:phret_intro} exactly and must resort to numerical optimization. Using $\mu$-sharpness, i.e., \eqref{eq: sharp_phase-retrieval}, we have that any $\epsilon$-suboptimal solution 
	to \eqref{eq:phret_intro} or \eqref{eq:phret_penalized_intro} satisfies $\norm{X - X^\natural}_1 \leq \epsilon/\mu$. Thus, the parameter $\mu$ allows us to control how optimization error affects the recovery process.
	\item \emph{Measurement error}: Suppose that instead of observing $b = \cA(X^\natural)$, we observe $\tilde b= \cA(X^\natural) + \delta$. 
	It is known that \eqref{eq:phret_intro} actually admits a unique \emph{feasible} solution $X^\natural$ w.h.p.\ \cite{candes2014solving}, thus 
	\eqref{eq:phret_intro} with $b$ replaced by $\tilde b$ is likely to be infeasible. On the other hand, 
given any optimizer $\tilde X$ of \eqref{eq:phret_penalized_intro} with $\tilde b$ in place of $b$, we may use sharpness in the noiseless case to bound $\norm{\tilde X - X^\natural}_1 \leq O(\norm{\delta}_1/\mu)$ in the noisy case (see \cref{prop:robust_noise}).

	In a similar vein, we may consider a setting where $b$ is corrupted by a sparse vector $\delta$. Sharpness in the noiseless case tells us that the optimizer of \eqref{eq:phret_penalized_intro} is \emph{unchanged} if $\abs{\supp(\delta)}/m = O(\mu)$, i.e., if up to an $O(\mu)$ fraction of the entries of $b$ are corrupted (see \cref{prop:robust_sparse,rem:sparse_noise}).

    Thus, the parameter $\mu$ allows us to control how measurement error affects the recovery process.

	\item \emph{Algorithms}: Since \eqref{eq:phret_intro} is well-conditioned in $\ell_1$,
    it necessarily gains a dimension dependence and becomes ill-conditioned in the standard $\ell_2$ norm.
    Specifically, \cref{prop:poor_ell_2} implies that if \eqref{eq:phret_intro} is $\Theta(1)$-sharp and $\Theta(1)$-Lipschitz in the $\ell_1$ norm, then it must be $O(1)$-sharp and $\Omega(\sqrt{n})$-Lipschitz in the $\ell_2$ norm.
    This ill-conditioning causes both 
    the theoretical performance guarantees and the numeric performance of standard sharpness-aware first-order methods, which are based on the $\ell_2$ norm, to deteriorate as the dimension increases on problem \eqref{eq:phret_intro}.
    For example, subgradient descent with Polyak stepsizes,
    which is optimal on the class of $L_2$-Lipschitz, $\mu_2$-sharp functions in the $\ell_2$ norm, can at best guarantee a convergence rate of 
    \begin{align*}
        O\left(\kappa^2 n\log\left(\frac{1}{\epsilon}\right)\right)
    \end{align*}
    on \eqref{eq:phret_intro} (see \cref{subfig:polyak_rgd}).
    Thus, to solve large-scale instances of \eqref{eq:phret_intro}, we develop a new restarted mirror descent algorithm (RMD), which can be applied to \emph{general} convex functions that 
	are $\mu$-sharp and $L$-Lipschitz in terms of an $\ell_p$ norm with $p\in[1,2]$ (see \Cref{sec:algs}).
When applied to \eqref{eq:phret_penalized_intro}, we have $p=1$ and RMD produces an $\epsilon$-suboptimal solution in
	\begin{align*}
		O\left(\kappa^2\log(n)\log\left(\frac{1}{\epsilon}\right)\right)
	\end{align*}
	iterations.
    This convergence rate depends only logarithmically on $n$. This contrasts with the linear dependence on $n$ that an $\ell_2$-based algorithm would incur.
	These convergence guarantees also hold after appropriate modifications in the presence of corruption or noise (see \cref{prop:robust_noise,prop:robust_sparse}) with guarantees depending on $\kappa$. 
\end{itemize}

We note that our results on optimization error, measurement error, and algorithms above depend solely on the deterministic assumption of sharpness and the conditioning $(\mu,r,\ell, L)$. The statistical assumptions are only used to prove that the conditioning $(\mu,r,\ell, L)$ is well-behaved w.h.p. We emphasize again that the above results continue to hold similarly for other modern signal recovery problems.

\subsection{Related work}
To better position our work in the literature, in this section, we discuss related work on error bounds and well-conditioning, 
sharpness and well-conditioning in statistical recovery, and first-order methods for minimizing sharp functions.

\paragraph{Establishing error bounds and well-conditioning} The sharpness condition in \eqref{eq: sharp_phase-retrieval} can be viewed as an error bound for \eqref{eq:phret_intro}.
Error bounds and conditioning are 
central to optimization problems both structurally and algorithmically \cite{alizadeh1997complementarity,lewis2014nonsmooth,zhou2017unified,drusvyatskiy2018error,necoara2019linear,johnstone2020faster}.
Consider the following general 
error bound for \eqref{eq:phret_intro}: 
\begin{equation*}\tr(X)+ r \norm{\cA(X) - b}_1 + \ell \dist_{\norm{\cdot}}(X,\S^n_+) - 
	\tr(X^\natural) \geq \mu \norm{X-X^\natural}^h, \quad \text{for all}\; X\in\S^n,
\end{equation*}
where, in addition to the $r$ and $\ell$ introduced in \eqref{eq: sharp_phase-retrieval}, we also have the power $h$, and 
a general norm $\norm{\cdot}$. Most methods for establishing error bounds either require that \eqref{eq:phret_intro} is linear \cite{hoffman2003approximate} or that regularity conditions such as Slater's condition \cite{bauschke1999strong,zhang2000global}, strict complementarity \cite{sturm2000error,ding2022strict}, or nondegeneracy conditions \cite{nayakkankuppam1999conditioning} hold. Additional work has studied error bounds in generic semialgebraic settings \cite{drusvyatskiy2016generic}. Results 
of this type
generally place an emphasis on establishing the power of $h$ (usually to $1$ or $2$) and provide only
weak bounds on $\mu$, $r$, and $\ell$. 
This is natural as the power $h$
determines the linear or sublinear convergence rate \cite{attouch2010proximal,johnstone2020faster}. On the other hand, if $\mu$, $r$, or $\ell$ depend polynomially on the dimension, then first-order algorithms may still be doomed to poor or dimension-dependent convergence rates.
This work connects the quantities $\mu$, $r$, and $\ell$ to the quantitatively better-understood restricted isometry property (RIP).
This connection gives precise estimates of ($\mu,r,\ell,L)$ in our settings and shows that these quantities are in fact dimension independent.

\paragraph{Sharpness and well conditioning in statistical recovery} A series of works \cite{duchi2019solving,li2020nonconvex,charisopoulos2019composite,charisopoulos2021low} studied the local landscape of \emph{nonconvex} nonsmooth formulations in statistical recovery near the signal, particularly in low-rank matrix recovery. They 
showed that near a factored form of the ground truth $X^\natural$, the natural loss function is sharp, and the local condition number (in terms of the Frobenius or Schatten-$2$ norm) does not depend on the dimension directly. Hence, they can apply off-the-shelf algorithms, such as subgradient and prox-linear methods \cite{davis2018subgradient,drusvyatskiy2019efficiency}  to achieve quick convergence to a factored form of $X^\natural$. Note that results of this type are local, and the algorithms require specific model knowledge to construct an appropriate initializer to achieve provable guarantees. In contrast, the methods developed here are convex optimization methods so they are not as sensitive to poor initialization.

\paragraph{First-order methods for minimizing sharp functions}
Early work in this area established linear convergence of the subgradient method (with appropriate step sizes) for $\mu$-sharp $L$-Lipschitz functions in the Euclidean norm \cite{eremin1965relaxation,polyak1969minimization,goffin1977convergence}.
These methods and their proofs are adapted to the Euclidean norm and incur a \emph{linear dependence on the dimension} when applied to $\mu$-sharp $L$-Lipschitz functions in $\ell_1$ or Schatten-1 norms.

Similar guarantees can be derived as a consequence of restarting schemes. Perhaps the earliest work discussing restarting schemes is \cite{nemirovski1985optimal}, where a restarting scheme was developed for convex H\"older-smooth minimization in $\ell_p$ spaces. More recent work~\cite{yang2018rsg,roulet2017sharpness,roulet2015renegar,necoara2019linear}
has used restarting schemes to accelerate variants of (sub)gradient descent in the presence of different growth conditions, e.g., a local guarantee of the form\footnote{Other forms of the necessary ``growth'' are possible in different settings.}
\begin{align*}
	f(x)-\min_x f(x) \geq \mu \dist(x,\cX)^h
\end{align*}
for all $x$ close enough to the set of minimizers, $\cX$.
Here, $\dist(\cdot,\cdot)$ is typically measured in the Euclidean norm and $h\in[1,\infty)$ captures the H\"olderian growth of $f$.
Taking $h=1$ recovers the usual definition of sharpness. In this setting, restarted subgradient descent achieves a linear convergence rate of $O\left(\kappa^{2}\log(1/\epsilon)\right)$ \cite{yang2018rsg}.
Adaptive versions of these restart schemes have also been developed~\cite{roulet2017sharpness,renegar2022simple} that do not require the sharpness parameter $\mu$ to be specified but incur an additional $\log(1/\epsilon)$ factor in the convergence rate.

We emphasize that much of the preceding literature on first-order methods for sharp convex functions focuses on the Euclidean case. The only exception we know of is the work of \cite{roulet2017sharpness}, which develops a restarted mirror descent method for functions satisfying a property inspired by sharpness in the Euclidean setting. Roulet and d'Aspremont~\cite{roulet2017sharpness} show that their method can be used to achieve $O(\kappa^2\log(n)\log(1/\epsilon))$ convergence on this particular function class.
Unfortunately, their proxy for sharpness is quite restrictive and does not hold in settings where sparse or low rank solutions are expected.
Our work in \cref{sec:algs} is closely related to \cite{roulet2017sharpness} but uses a different specification of the restarting mechanism. This change allows us to use a more natural definition of sharpness that is also much broader (see \cref{rem:roulet_comparison} for a thorough comparison).

\subsection{Outline and notation}
\paragraph{Outline} We begin, in \Cref{sec:sharpness}, by setting up notation and defining sharpness for an abstract signal recovery problem.
We will quantify sharpness and conditioning by a set of parameters $(\mu,r,\ell,L)$.
In \Cref{sec:sharpness_examples,sec:Sharpnessconditioningandsamplecomplexity}, we establish well-conditioning and sharpness under the Gaussian sensing models 
for the problems of sparse recovery, low-rank matrix and bilinear sensing, and covariance estimation in the absence of noise.
Our proof
relies on the well-known restricted isometry property (RIP).
In \cref{sec:stat}, we investigate how sharp instances of the abstract signal recovery problem behave in the presence of noise. 
In \cref{sec:algs}, we describe and analyze an algorithm restarted mirror descent (RMD) that has linear and nearly dimension-independent convergence rates when applied to sharp functions in $\ell_p$ or Schatten-$p$ norms (see \cref{as:lp_setup}). In conjunction with the results of \Cref{sec:stat}, we may apply RMD to obtain linear nearly dimension-independent convergence rates even in the presence of noise for these statistical recovery problems.

\paragraph{A note on $\ell_p$ and Schatten-$p$ norms}
Let $p\in[1,\infty]$. We will overload notation and use $\norm{x}_p$ for the $\ell_p$ norm when $x\in\R^n$ and the Schatten-$p$ norm when $x\in\R^{n\by N}$. Recall that the nuclear norm is equivalent to the Schatten-$1$ norm, the Frobenius norm is equivalent to the Schatten-$2$ norm, and the spectral norm is equivalent to the Schatten-$\infty$ norm. To streamline the text, we will also take ``$\ell_p$ norm'' of a matrix to mean its Schatten-$p$ norm.

\section{Preliminaries}
\label{sec:sharpness} 
In this section, we describe an abstract signal recovery problem and the convex optimization approach. Then, we define our measures of sharpness and conditioning.
\subsection{Signal recovery via convex optimization}
Let $V$ and $W$ be normed Euclidean spaces. In the abstract signal recovery problem, we aim to recover an unknown element $x^\natural\in V$, called the signal, from a linear sensing operator, $\cA:V\to W$, and a vector of measurements, $b = \cA(x^\natural) \in W$.
The convex optimization approach solves the following generalization of \eqref{eq:phret_intro},
\begin{align}\tag{P}
    \label{eq:abstract_problem}
    \min_{x\in V}\set{f(x):\, \begin{array}{l}
    \cA(x) = b\\
    x\in \cK
    \end{array}},
\end{align}
to recover $x^\natural$.
Here, the function $f:V\to\R$ is convex  and $\cK\subseteq V$ is a closed convex cone (possibly all of $V$). They reflect our prior knowledge of $x^\natural$, e.g., if $x^\natural$ is sparse, we set $f(x) = \|x\|_1$. 

\paragraph{A note on the use of norm} Each Euclidean space $V$ and $W$ comes with a norm $\norm{\cdot}$ that may be different from the Euclidean norm. For notational simplicity, we write all norms as $\norm{\cdot}$ with the \emph{understanding that the norm is the norm associated with the space of the argument}. 
For example if $x\in V$, then $\norm{x}$ is the norm of $x$ measured in the norm of $V$. 

Next, we define the sharpness and conditioning of \eqref{eq:abstract_problem}.

\subsection{Definition of sharpness and conditioning}

For $x\in V$ and a closed nonempty set $\cX\subseteq V$, let
\begin{align*}
	\dist(x,\cX) \coloneqq \min_{\bar x\in \cX}\norm{\bar x - x}.
\end{align*}
Again, the norm $\norm{\bar x - x}$ is measured in the space of the argument. In this case, the expression $\bar x -x \in V$ is measured in the norm of $V$.
The following definition of sharpness for a convex function (or unconstrained minimization problem) is also referred to as weak sharpness in the literature~\cite{ferris1988weak}.
\begin{definition}[Sharpness]
	\label{def:sharp_fn}
	Let $V$ be a normed Euclidean space.
	Let $f:V\to\R$ be a convex function, $\cX\subseteq V$ a nonempty convex set, and let $\mu>0$.
	We say that $f$ is $\mu$-sharp around $\cX$ if $\cX = \argmin f$ and
	\begin{align*}
		f(\bar x) - \min_{x\in V}f(x) &\geq \mu \dist(\bar x, \cX),\qquad\forall \bar x\in V.\qedhere
	\end{align*}
\end{definition}

We extend the notion of sharpness to a problem of the form \eqref{eq:abstract_problem} as follows.

\begin{definition}[Sharpness of  \eqref{eq:abstract_problem}]
	\label{def:sharpness}
	Consider a problem of the form \eqref{eq:abstract_problem}. Let $x^\natural\in V$, $\mu>0$, and $r,\,\ell\geq 0$. We say that \eqref{eq:abstract_problem} is \emph{$(\mu,r,\ell)$ sharp around $x^\natural$} if $x^\natural$ is feasible in \eqref{eq:abstract_problem} and
	\begin{align}\label{eq: MainPenalizationForm}
		F_{r,\ell}(x) \coloneqq f(x) + r\norm{\cA(x) - b} + \ell \dist(x,\cK)
	\end{align}
	is $\mu$-sharp around $x^\natural$. Equivalently, if $x^\natural$ is feasible in \eqref{eq:abstract_problem} and
	\begin{align}
		\label{eq:sharp_as_error_bound}
		\norm{x - x^\natural} \leq \frac{1}{\mu}\left(f(x) - f(x^\natural) + r \norm{\cA(x) - b} + \ell \dist(x,\cK)\right),\qquad\forall x\in V.
	\end{align}
	If $\cK=V$, then $\ell$ is inconsequential, so we will simply say \eqref{eq:abstract_problem} is $(\mu,r)$ sharp around $x^\natural$.
\end{definition}

\begin{remark}
	\label{rem:error_bounds}
   Suppose \eqref{eq:abstract_problem} is $(\mu,r,\ell)$-sharp around $x^\natural$ for some $\mu>0$, $r,\ell\geq 0$.
   Note that $x^\natural$ is necessarily the unique optimal solution of \eqref{eq:abstract_problem}. Moreover, we have an error bound: If we numerically solve \eqref{eq:abstract_problem} and produce $\tilde x\in V$, then we may bound the distance $\norm{\tilde x - x^\natural}$ according to \eqref{eq:sharp_as_error_bound}.
	This bound holds even if $f(\tilde x) < f(x^\natural)$, which could happen as $\tilde x$ is  not necessarily feasible for \eqref{eq:abstract_problem}.
\end{remark}

Recall the standard definition of the Lipschitz constant of a function. 

\begin{definition}[Lipschitz constant]
We say that a convex function $f:V\to\R$ is $L$-Lipschitz if 
\begin{align*}
	\abs{f(x) - f(y)} &\leq L \norm{x - y},\qquad\forall x, y\in V.\qedhere
\end{align*}
\end{definition}

We are now ready to define the conditioning of \eqref{eq:abstract_problem}. Note that the following definition 
depends on the norms associated with $V$ and $W$.
\begin{definition} [Conditioning and condition number]
Consider a problem of the form \eqref{eq:abstract_problem} and suppose $x^\natural\in V$ is its unique optimizer. Suppose for some $r,\ell\geq 0$ that $F_{r,\ell}$ is $\mu$-sharp around $x^\natural$ and $L$-Lipschitz, with $\mu,L>0$. The conditioning of \eqref{eq:abstract_problem} is measured by 
$(\mu,r,\ell,L)$ and the condition number is $\kappa = \frac{L}{\mu}$.
\end{definition}

\section{RIP-based sharpness and conditioning in signal recovery}
\label{sec:sharpness_examples}

This section shows that under bounds on restricted isometry (defined below), sharpness holds for three archetypal modern signal recovery problems:
sparse vector recovery, low rank matrix recovery, and covariance estimation.
Phase retrieval can also be thought of as a special case of covariance estimation. 
The signal to be recovered, the optimization 
formulation in \eqref{eq:abstract_problem}, and the choices of norms 
are the following.

\begin{definition}\label{def:optChoice}
	We consider the following signal recovery problems that can be formulated as \eqref{eq:abstract_problem} and the corresponding choices for 
	$V$, $f$, and $\cK$. 
\begin{itemize}
	\item Sparse vector recovery: recover a $k$-sparse vector $x^\natural \in V = \R^n$. In the formulation \eqref{eq:abstract_problem}, 
	the objective $f(x) = \|x\|_1$, the set $\cK = \R^n$, and the input space $V = (\R^n,\norm{\cdot}_1)$.
	\item Low rank matrix recovery
:  recover a rank $k$ matrix $X^\natural \in V = \R^{n\times N}$. In the formulation \eqref{eq:abstract_problem}, 
	the objective $f(X) = \|X\|_1$, the set $\cK = \R^n$, and the input space $V = (\R^{n\times N},\norm{\cdot}_1)$.
	\item Covariance estimation: recover a rank $k$ PSD matrix  $X^\natural \in V = \S^{n}$. In the formulation \eqref{eq:abstract_problem}, 
	the objective is the trace $f(X) = \tr (X)$, the set $\cK$ is the set of PSD matrices $\S^{n}_+$, and the input space $V = (\S^{n},\norm{\cdot}_1)$.\qedhere
\end{itemize}
\end{definition} 

The space $W$ and its norm are, as yet, unspecified.
We have some freedom in the choice of the norm as long as crucial bounds on the
restricted 
isometry property (RIP) hold in the norm on $W$.

\begin{definition}[Restricted isometry property]
Let $k'$ be a positive integer and $\cA:V\to W$ be a linear operator. We will let $\ripdown_{k'}(\cA)$ and $\ripup_{k'}(\cA)$ denote any valid \emph{uniform} lower bound and upper bound on the quantity $\norm{\cA x}/\norm{x}_2$
as $x\in V$ ranges over all elements in $V$ with support size or rank at most $k'$.
\end{definition}

Our main result of this section is the following:
	\begin{proposition}
		\label{prop:genSharpnessAndConditioning}
        Consider one of the three problems defined in \cref{def:optChoice}.
        Let $c = 1$ for sparse vector recovery and $c=2$ for the other two cases.
        Let $k'>0$ and $\epsilon>0$ and suppose 
        $\ripup_{k'}(\cA)\geq 1$ and 
		\begin{align*}\sqrt{\frac{k'}{ck}}\left(\frac{\ripdown_{ck+k'}(\cA)}{\ripup_{k'}(\cA)}\right) \geq 1 + \epsilon.
		\end{align*}
		Then, \eqref{eq:abstract_problem} is $\left(\frac{\epsilon}{2+\epsilon}, \sqrt{k'},2\right)$ sharp around $x^\natural$. Furthermore, the Lipschitz constant $L$ is no more than $3+\sqrt{k'}\ripup_{1}(\cA)$.
\end{proposition}

To put this result in context, we will see in \Cref{sec:Sharpnessconditioningandsamplecomplexity} that the premise of \Cref{prop:genSharpnessAndConditioning} holds once the sample size $m$ is a small multiple of the recovery threshold w.h.p.\ with $\epsilon = 2$ and $k' = O(k)$, and $\ripup_1(\cA)=O(1)$.
Hence, Problem \eqref{eq:abstract_problem}  is  $\left(\frac{1}{2},O\left(\sqrt{k}\right),2\right)$ sharp around $x^\natural$ with a Lipschitz constant no more than $O\left(\sqrt{k}\right)$ w.h.p.

The rest of the section consists of proofs of 
\Cref{prop:genSharpnessAndConditioning} for the three different settings. We 
start with the proof of sharpness of matrix sensing and discuss how this proof 
can be modified for the other two settings. We then bound the Lipschitz constant for all three settings.

\subsection{Proof of sharpness in matrix sensing}
\label{subsec:mat_sens}
This section proves \cref{prop:genSharpnessAndConditioning} for the setting of low-rank matrix sensing and uses many elements of the proofs for exact recovery \cite{candes2005decoding,recht2010guaranteed}.

	Let $\Delta\in\R^{n\by N}$ be arbitrary. Our goal is to show that
	\begin{equation}\label{eq: sharpness_goal}
		\norm{X^\natural + \Delta}_1 + \sqrt{k'} \norm{\cA(\Delta)} - \norm{X^\natural}_1 \geq \frac{\epsilon}{2+\epsilon}\norm{\Delta}_1.
	\end{equation}
	
\paragraph{Change of basis} Without loss of generality, we may work in a basis such that
	\begin{align*}
		X^\natural = \begin{pmatrix}
			X^\natural_{1,1} & 0_{k\by (N-k)}\\
			0_{(n-k)\by k} & 0_{(n-k)\by (N-k)}
		\end{pmatrix}
		\qquad\text{and}\qquad
		\Delta = \begin{pmatrix}
			\Delta_{1,1} & \Delta_{1,2}\\
			\Delta_{2,1} & \Delta_{2,2}
		\end{pmatrix},
	\end{align*}
	where $X^\natural_{1,1}$ is a $k\by k$ diagonal matrix and $\Delta_{2,2}\in\R^{(n-k)\by (N-k)}$ is diagonal with diagonal entries that are nonincreasing in magnitude.
	
\paragraph{Partitioning of the difference $\Delta$}	Let $\delta_{k^\perp}\coloneqq \diag(\Delta_{2,2})$. The rest of the proof is 
	based on the idea in \cite{candes2005decoding,recht2010guaranteed}.
	We decompose $\delta_{k^\perp} = \sigma_1 + \dots + \sigma_t$ where each $\sigma_i \in\R^{\min(n,N) - k}$. Specifically, let $\sigma_1$ extract the first $k'$ coordinates of $\delta_{k^\perp}$, let each subsequent $\sigma_i$ extract the next $k'$ coordinates of $\delta_{k^\perp}$. Finally, $\sigma_t$ may have fewer than $k'$ coordinates of $\delta_{k^\perp}$.
	Let $\Sigma_i$ denote the matrix of size ${n\by N}$ with 
	$\Diag(\sigma_i)$ in its bottom-right $(n-k)\by (N-k)$ block.
	
\paragraph{Lower bound on $\norm{\Delta_{2,2}}_1$} Since $\sigma_i$ are disjoint, we can lower bound $\norm{\Delta_{2,2}}_1 = \norm{\delta_{k^\perp}}_1$ by 
\begin{equation*}
		\norm{\Delta_{2,2}}_1 =
	\norm{\delta_{k^\perp}}_1 
	\geq \sum_{i=1}^{t-1} \norm{\sigma_i}_1= k'\sum_{i=1}^{t -1} \frac{\norm{\sigma_i}_1}{k'}.
\end{equation*}
Using $\delta_{k^\perp}\text{ is nonincreasing}$ and each $\sigma_i$ is $k'$ sparse, we further have 
	\begin{equation*}
		\norm{\Delta_{2,2}}_1	\geq k'\sum_{i=2}^{t}\norm{\sigma_i}_\infty 
		\geq \sqrt{k'}\sum_{i=2}^t \norm{\sigma_i}_2.
		\end{equation*} 
	Next, we use the RIP to obtain
	\begin{equation*} 
		\norm{\Delta_{2,2}}_1	\geq \frac{\sqrt{k'}}{\ripup_{k'}}\sum_{i=2}^t \norm{\cA(\Sigma_i)}
\overset{(a)}{\geq} \frac{\sqrt{k'}}{\ripup_{k'}}\left(\norm{\cA\begin{pmatrix}
				\Delta_{1,1} & \Delta_{1,2}\\
				\Delta_{2,2} & 0
			\end{pmatrix} + \cA(\Sigma_1)} - \norm{\cA(\Delta)}\right),
		\end{equation*}
where $(a)$ is due to triangle inequality. Using the RIP condition again and  $\ripup_{k'}\geq 1$, 
we have 
	\begin{align*} 
		\norm{\Delta_{2,2}}_1 &	\geq \sqrt{k'} \frac{\ripdown_{2k+k'}}{\ripup_{k'}}\norm{\begin{pmatrix}
				\Delta_{1,1} & \Delta_{1,2}\\
				\Delta_{2,1} & 0
		\end{pmatrix} + \Sigma_1}_2 - \sqrt{k'}\norm{\cA(\Delta)} \\
		& \geq \sqrt{k'} \frac{\ripdown_{2k+k'}}{\ripup_{k'}}\norm{\begin{pmatrix}
				\Delta_{1,1} & \Delta_{1,2}\\
				\Delta_{2,1} & 0
		\end{pmatrix}}_2 - \sqrt{k'}\norm{\cA(\Delta)}. 
	\end{align*}
Lastly, since $\begin{pmatrix}
	\Delta_{1,1} & \Delta_{1,2}\\
	\Delta_{2,1} & 0
\end{pmatrix}$ has rank no more than $2k$, we have 
\begin{align*}
	\norm{\Delta_{2,2}}_1		&\geq \sqrt{\frac{k'}{2k}}\frac{\ripdown_{2k+k'}}{\ripup_{k'}}\norm{\begin{pmatrix}
				\Delta_{1,1} & \Delta_{1,2}\\
				\Delta_{2,1} & 0
		\end{pmatrix}}_1 - \sqrt{k'}\norm{\cA(\Delta)}\\
		&\geq (1+\epsilon)\norm{\begin{pmatrix}
				\Delta_{1,1} & \Delta_{1,2}\\
				\Delta_{2,1} & 0
		\end{pmatrix}}_1 - \sqrt{k'}\norm{\cA(\Delta)}.
	\end{align*}
	
\paragraph{Putting things together}	We are now ready to prove sharpness:
	\begin{align*}
		& \norm{X^\natural + \Delta}_1 + \sqrt{k'}\norm{\cA(\Delta)} - \norm{X^\natural}_1\\
		\qquad \geq & \norm{\begin{pmatrix}
				X_{1,1}^\natural & 0\\
				0 & \Delta_{2,2}
		\end{pmatrix}}_1-\norm{X^\natural}_1 - \norm{\begin{pmatrix}
				\Delta_{1,1} & \Delta_{1,2}\\
				\Delta_{2,1} & 0
		\end{pmatrix}}_1\\
        & +\sqrt{k'}\norm{\cA(\Delta)}\\
\geq & \norm{\delta_{k^\perp}}_1 - \norm{\begin{pmatrix}
				\Delta_{1,1} & \Delta_{1,2}\\
				\Delta_{2,1} & 0
		\end{pmatrix}}_1 +\sqrt{k'}\norm{\cA(\Delta)} && (\text{disjoint supports})\\
	= &  \left(\frac{2}{2+\epsilon} + \frac{\epsilon}{2+\epsilon}\right)\norm{\delta_{k^\perp}}_1 - \norm{\begin{pmatrix}
				\Delta_{1,1} & \Delta_{1,2}\\
				\Delta_{2,1} & 0
		\end{pmatrix}}_1 \\
    & +\sqrt{k'}\norm{\cA(\Delta)}\\
	\geq &  \left(\frac{2}{2+\epsilon}(1+\epsilon) - 1\right)\norm{\begin{pmatrix}
				\Delta_{1,1} & \Delta_{1,2}\\
				\Delta_{2,1} & 0
		\end{pmatrix}}_1 + \frac{\epsilon}{2+\epsilon}\norm{\delta_{k^\perp}}_1 && \text{(bound on $\norm{\delta_{k^\perp}}_1$)}\\
&+ \left(1 - \frac{2}{2+\epsilon}\right)\sqrt{k'}\norm{\cA(\Delta)}\\
		\geq & \frac{\epsilon}{2+\epsilon}\norm{\Delta}_1.
	\end{align*}
	This proves the claim \eqref{eq: sharpness_goal} as $\Delta\in \R^{n\by N}$ was arbitrary.

\subsection{Sharpness in sparse vector recovery}
\label{subsec:sparse_recovery}

One can follow the proof for matrix sensing almost verbatim by treating the signal $x^\natural$ 
	and the sensing vectors in $\cA$ as diagonal matrices. The change from $c=2$ to $c=1$ is justified by
	noting there is no extra off-diagonal term for the matrices considered in the argument. Thus, the matrix $\begin{smallpmatrix}\Delta_{1,1} & \Delta_{1,2}\\ \Delta_{2,1} & 0\end{smallpmatrix} = \begin{smallpmatrix}\Delta_{1,1} & \\ & 0 \end{smallpmatrix}$
    has rank bounded by $k$ instead of $2k$.

\subsection{Sharpness in covariance estimation}
\label{subsec:covariance}

We may repeat the proof of \Cref{subsec:mat_sens} verbatim after replacing $V = (\R^{n\by N},\norm{\cdot}_1)$ by $V = (\S^n,\norm{\cdot}_1)$. We deduce that, under the assumptions of this proposition,
\begin{align*}
\norm{X}_1 + \sqrt{k'}\norm{\cA(X) - b}
\end{align*}
is $\frac{\epsilon}{2+\epsilon}$ sharp around $X^\natural$.
Next, note that $\tr(X) + 2\dist(X, \S^n_+)\geq \norm{X}_1$. We conclude that
\begin{align*}
    \tr(X) + 2\dist(X, \S^n_+) +\sqrt{k'}\norm{\cA(X) - b}
\end{align*}
is $\frac{\epsilon}{2+\epsilon}$ sharp around $X^\natural$.

\subsection{Proof of Lipschitz continuity}
\label{subsec:lipschitz}

We need to show that 
\begin{equation}\label{eq: lipschitz_claim}
    f(x) + \sqrt{k'}\norm{\cA(x) -b} + 2\dist(x,K)
\end{equation}
is $3+\sqrt{k'}\ripup_1(\cA)$ Lipschitz with respect to $\norm{\cdot}_1$. 

By the triangle inequality, the objective functions $f(x) = \norm{x}_1$ and $f(X) = \tr(X)$ are both $1$-Lipschitz in terms of $\norm{\cdot}_1$.
In the covariance estimation setting, the distance function $\dist(X,\mathbb{S}_+^n)$ is also $1$-Lipschitz in terms of $\norm{\cdot}_1$ by the triangle inequality.

Next, we show the function $\norm{\cA(x)-b}$ is $\ripup_1$ Lipschitz. Indeed, for any $x_1,x_2\in V$, let $\Delta = x_1-x_2 = \sum_{i=1}^n \Delta_i$, where $\Delta_i$ are $1$-sparse or has rank no more than $1$ (due to singular value decomposition). We have 
\begin{align*}
&|\norm{\cA (x_1)-b} - \norm{\cA(x_2)-b}| \\
\leq & \norm{\cA(x_1 - x_2)} = \norm{\cA \left(\sum_i^n \Delta_i\right)} \leq \sum_{i=1}^n \norm{\cA \Delta_i} \\
\overset{(a)}{\leq} & 
\sum_{i=1}^n \ripup_1 \norm{\Delta_i}_2 \overset{(b)}{=} \sum_{i=1}^n \ripup_1 \norm{\Delta_i}_1 = \ripup_1 \norm{\Delta}_1.
\end{align*}
Here step $(a)$ and $(b)$ are due to the RIP and $\Delta_i$, $i=1,\dots,n$ are 1-sparse or have rank no more than $1$. 

Since $f(x)$ and $\dist(x,K)$ are $1$ Lipschitz and $\norm{\cA(x)-b}$ is $\ripup_1$ Lipschitz, our proof that the function in \eqref{eq: lipschitz_claim} is $3+\sqrt{k'}\ripup_1(\cA)$-Lipschitz is complete.

\section{Conditioning, sensing models, and sample complexity}
\label{sec:Sharpnessconditioningandsamplecomplexity}
In this section, we describe different sensing models of $\cA$ and show that once $m$ exceeds certain thresholds, Problem \eqref{eq:abstract_problem}  is well-conditioned in terms of $\ell_1$ norm with $W = (\R^m,\norm{\cdot}_p)$ where $p=1,2$. 

To prove this result, we first describe the precise sensing model in \Cref{sec: sensing_model_sample_complexity}. Then we collect bounds on the RIP from the literature that ensure that the premise of \Cref{prop:genSharpnessAndConditioning} is satisfied. Due to an additional technicality in one of the sensing models (labeled ``Covariance Estimation I'' below), we give a separate proof of its well-conditioning in \Cref{sec: proof of covariance estimation}. 

The sensing map $\cA$ is described by its adjoint map $\cA^*$. More precisely, we equip $\mathbb{S}^n$ and $\mathbb{R}^{n \times N}$ with the trace inner product and equip $\mathbb{R}^n$ and $\mathbb{R}^m$ with the dot product. Then the adjoint map $\cA^*$ is well-defined and we decribe how $\cA^*(e_i) \in V$, $i=1,\dots,m$ is generated where $e_i$ is the $i$-th standard basis of $\R^m$.

A summary of the sensing model, the thresholds, and the norms on $W = \R^m$ is described in Table \ref{tb: TaskRIPmCondition}. Our main result of this section is the following theorem. 

\begin{theorem}
	\label{thm: sampleVSconditioning}
Suppose the space $W$ and the sampling map $\cA$ are described according to one of the scenarios in Table \ref{tb: TaskRIPmCondition} and
 the signal $x^\natural$ is either $k$-sparse or has rank no more than $k$. If $m\ge CT(n,N,k)$, where 
 $T(n,N,k)$ is defined in Table \ref{tb: TaskRIPmCondition} and $C$ is a numerical constant, then
 there are numerical constants $c_1,c_2>0$  such that  
w.h.p.\ the optimization problem \eqref{eq:abstract_problem} is $(\frac{1}{2}, \sqrt{c_1k}, 2)$-sharp
around $x^\natural$ and $F_{\sqrt{c_1k},2}$ has a Lipschitz constant bounded by $\sqrt{c_2k}$. Consequently, 
\eqref{eq:abstract_problem} has a condition number $\kappa$ bounded by $2\sqrt{c_2k}$. 
\end{theorem}

\begin{table}[h]
    \centering
	\begin{tabular}{|c|c|c|c|} 
		\hline
		Task & $\cA^*(e_i)$ & $T(n,N,k)$ & norms of $W=\R^m$\\
		\hline
		\hline 
	Sparse vector recovery& $a_i \in \R^n$ & $k\log(n/k)$ & $\ell_1$ or $\ell_2$  \\ 
	Low rank matrix sensing I& $A_i\in \R^{n\times N}$ & $\max\{n,N\}k$ & $\ell_1$ or $\ell_2$\\ 
	Low rank matrix sensing II & $a_ib_i^\top\in \R^{n\times N}$ & $\max\{n,N\}k$&$\ell_1$  \\ 
	Covariance estimation I& $a_ia_i^\top\in \S^{n}$ & $nk$ & $\ell_1$ \\ 
	Covariance estimation II& $a_ia_i^\top -b_ib_i^\top \in \S^n$ & $nk$ & $\ell_1$ \\ 
		\hline
	\end{tabular}
\caption{Description of different statistical signal recovery tasks and the thresholds for well-conditioning. The entries of $A_i, a_i,$ and $b_i$ are i.i.d. 
Gaussian random variables with appropriate scaling (see Section \ref{sec: sensing_model_sample_complexity}). The conditioning of \eqref{eq:abstract_problem} is measured by $(\frac{1}{2},\sqrt{c_1k},2,\sqrt{c_2k})$ where $c_1,c_2>0$ are numerical constants.}\label{tb: TaskRIPmCondition}
\end{table}

\subsection{Sensing model, sample complexity, and proof via RIP}
\label{sec: sensing_model_sample_complexity}

In this section, we first describe the sensing models and the thresholds $T(n,N,k)$ on $m$ so that bounds on RIP hold for $\cA$ when $W=\R^m$ is equipped with either the $\ell_1$ or $\ell_2$ norms. Then, we prove \Cref{thm: sampleVSconditioning} by verifying the premise of \Cref{prop:genSharpnessAndConditioning}.

Recall that $\ripdown_{k'}(\cA)$ and $\ripup_{k'}(\cA)$ are any uniform lower and upper bound on $\norm{\cA x}/\norm{x}_2$ as $x$ ranges over elements of $V$ with support or rank bounded by $k'$.
They will be set to numerical constants $c_1,c_2$ below and could differ for different sensing models.
The norm of $W = \R^m$ will be either the $\ell_1$ norm or $\ell_2$ norm below.

\paragraph{Sparse vector recovery} For sparse vector recovery, 
the measurements are of the form
\begin{equation*}
 \cA^*(e_i) = a_i, \quad i = 1,\dots,m,
\end{equation*}
where the measurement vectors $a_i\overset{\text{i.i.d.}}{\sim} N\left(0,I/m\right)$ in the $\ell_2$ setting and $a_i\overset{\text{i.i.d.}}{\sim} N\left(0,I/m^2\right)$ in the $\ell_1$ setting.
In both settings, we may set w.h.p.\ 
$\ripup_{k'}\left(\cA\right) = c_2$ and $\ripdown_{k'}\left(\cA\right) = c_1$ as long as $m\gtrsim k'\log (\frac{n}{k'})$,\footnote{We write $a\gtrsim b$ if $a\ge Cb$ for some numerical constant  $C>0$.}
where $c_1$ and $c_2$ are constants independent of $k'$ satisfying $c_2/c_1 \leq 1.1$ and $c_2\geq 1$.
This fact is proved in the $\ell_2$ setting following 
\cite{candes2006stable,edelman1988eigenvalues} using results on singular values of random matrices; a simple proof 
can be found in \cite[Theorem 5.2]{baraniuk2008simple}. 
This fact can be proved in the $\ell_1$ setting using \cite[Lemma 2.1]{plan2014dimension} and \cite[Lemma 4.4]{rudelson2006sparse}.  \footnote{Specifically, \cite[Note that Lemma 2.1]{plan2014dimension} provides a RIP bound with a quantity called Guassian width. One then use \cite[Lemma 4.4]{rudelson2006sparse} to estimate the width and obtain the concrete RIP above.}

\paragraph{Matrix sensing I} For this scenario of matrix sensing, 
the measurements are of the form
\begin{equation*}
	\cA^*(e_i) = A_i, \quad i = 1,\dots,m,
\end{equation*}
where each measurement matrix $A_i \in \R^{n\times N}$ has Gaussian entries. Each entry of each matrix is sampled i.i.d.\
according to $N\left(0, 1/m\right)$ in the $\ell_2$ setting, and according to $N\left(0,1/m^2\right)$ in the $\ell_1$ setting.
In both settings, we may set w.h.p.\  $\ripup_{k'}(\cA) = c_2$ and $\ripdown_{k'}(\cA) = c_1$ as long as $m \gtrsim k' \max(n,N)$,
where $c_1$ and $c_2$ are constants independent of $k'$ satisfying $c_2/c_1 \leq 1.1$ and $c_2\geq 1$.
This fact is proved in the $\ell_2$ setting
in \cite[Theorem 2.3]{candes2011tight}.
This fact is proved in the $\ell_1$ setting
in \cite[Proposition 1]{li2020nonconvex}.

\paragraph{Matrix sensing II} For this version of matrix sensing, which is more commonly known as 
bilinear sensing,  
the measurements are of the form
\begin{equation*}
	\cA^*(e_i) = a_i b_i^\top , \quad i = 1,\dots,m,
\end{equation*}
and the measurement vectors $a_i\overset{\text{i.i.d.}}{\sim} N(0,I_n/m)$ and $b_i\overset{\text{i.i.d.}}{\sim} N(0,I_N/m)$.
We equip $W=\R^m$ with the $\ell_1$ norm.
In this setting, we may set w.h.p.\ $\ripup_{k'}(\cA) = c_2$ and $\ripdown_{k'}(\cA) = c_1$ as long as 
$m\gtrsim k'\max(n,N)$,
where $c_1$ and $c_2$ and constants independent of $k'$ satisfying $c_2/c_1\leq 4$ and $c_2\geq 1$.
This fact is proved according to \cite[Theorem 2.2]{cai2015rop}.

\paragraph{Covariance estimation I} For this scenario of covariance estimation, 
the measurements are of the form
\begin{equation*}
	\cA^*(e_i) = a_i a_i^\top , \quad i = 1,\dots,m,
\end{equation*}
where each measurement vector $a_i\overset{\text{i.i.d.}}{\sim} N(0,I_n/m)$. 
We equip $W=\R^m$ with the $\ell_1$ norm.
The proof of \cref{thm: sampleVSconditioning} for this setting differs from the proof of \cref{thm: sampleVSconditioning} for all other settings.
This difference stems from the fact that
$\ip{\cA^*(e_i), X^\natural}$ does not have zero mean, thus biasing the output vector as 
discussed in \cite[Section III.B]{chen2015exact}.
Attempting to follow the same proof strategy will need $\ripup_{k'}(\cA)$ to be a numerical constant. However, the quantity $\ripup_{k'}(\cA)$ scales as $\sqrt{k'}$. This prevents us from applying \cref{prop:genSharpnessAndConditioning} directly in this setting.

We provide a separate proof for \cref{thm: sampleVSconditioning}(Covariance estimation I) in \Cref{sec: proof of covariance estimation}. The proof in this setting is completed by analyzing the conditioning of a related model of covariance estimation (Covariance estimation II).

\paragraph{Covariance estimation II} For this scenario of covariance estimation, 
the measurements are of the form
\begin{equation*}
	\cA^*(e_i) = a_i a_i^\top  - b_ib_i^\top, \quad i = 1,\dots,m,
\end{equation*}
where each measurement vector pair $a_i,b_i\overset{\text{i.i.d.}}{\sim} N(0,I_n/(2m))$. 
We equip $W=\R^m$ with the $\ell_1$ norm.
In this setting, we may set w.h.p.\ $\ripup_{k'}(\cA) = c_2$ and $\ripdown_{k'}(\cA) = c_1$ as long as  $m\gtrsim n k'$. Here, $c_1$ and $c_2$ are constants independent of $k'$ satisfying $c_2/c_1\leq 4$ and $c_2 \geq 1$. To prove this fact, consider the operator $\bar{\cA}: \mathbb{R}^{n\times n} \rightarrow W$ such that $\bar{\cA}^*(e_i) =  (a_i - b_i)(a_i+b_i)^\top$. Note that $a_i-b_i$ and $a_i+b_i$ are independent of each other and $a_i-b_i, a_i+b_i\overset{\text{i.i.d.}}{\sim} N(0,I_n/m)$ for $i=1,\dots,m$. Using  \cite[Theorem 2.2]{cai2015rop}, we see that RIP holds for $\bar{\cA}$ w.h.p. with $\ripup_{k'}(\bar{\cA}) = c_2$ and $\ripdown_{k'}(\bar{\cA}) = c_1$ as long as  $m\gtrsim n k'$. The proof is completed by noting that 
$[\bar{\cA}(X)]_i = \langle (a_i - b_i)(a_i+b_i)^\top,X\rangle = [\cA(X)]_i$ for any $X \in \S^{n}$.

The Gaussian assumption on the sensing vectors or matrices are assumed for simplicity. One can relax this condition
to sub-Gaussian distributions as done in  \cite[Theorem 5.2]{baraniuk2008simple}, \cite[Lemma 2.1]{plan2014dimension}, and \cite[Theorem 6.4]{charisopoulos2021low}.

\begin{proof}[Proof of Theorem \ref{thm: sampleVSconditioning} for all settings except covariance estimation I]
    As described above, for all settings except covariance estimation I, as long as $m\gtrsim T(n,N,k')$, we may set w.h.p.\ $\ripup_{k'}(\cA) = c_2$ and $\ripdown_{k'}(\cA) = c_1$. By setting $k' =\sqrt{C_1k}$ for a large numerical constant $C_1$, we see the premise of \Cref{prop:genSharpnessAndConditioning} is satisfied and our proof is complete.
\end{proof}

\subsection{Well-conditioning of covariance estimation I}
\label{sec: proof of covariance estimation}

We provide a separate argument for \cref{thm: sampleVSconditioning} in the setting of Covariance estimation I.

In this setting, we have that $a_i \sim N(0,I_n/m)$ are i.i.d. We equip the space $W= \R^m$ with the $\ell_1$ norm. Our goal is to show that \eqref{eq:abstract_problem} is sharp in terms of $\ell_1$ norm with parameters $(\frac{1}{2},\frac{1}{2}\sqrt{c_1 k},2)$ and Lipschitz continuous with Lipschitz constant $L= \sqrt{c_2k}$ w.h.p.\ once $m\gtrsim nk$. We will do so by comparing Covariance estimation I with Covariance estimation II.

First, we replace the linear constraint in Problem \eqref{eq:abstract_problem} by 
$\cB(X) = d$ where $\cB: \S^n \rightarrow \R^{\lfloor m /2 \rfloor}$, and
\begin{equation}\label{eq: combinedBd}
\cB^* (e_i) = \frac{1}{2} a_{2i-1}a_{2i-1}^\top - \frac{1}{2} 
a_{2i}a_{2i}^\top \quad d_i = \frac{1}{2}b_{2i-1}-
\frac{1}{2}b_{2i},\quad i = 1,\dots, \lfloor m/2 \rfloor.
\end{equation}
This is distributed as an instance of Covariance estimation II.
By Theorem \ref{thm: sampleVSconditioning} (Covariance estimation II), we know that once $m\gtrsim n k$, it holds that
\begin{equation}\label{eq: sharpCovII}
\tr(X) + \sqrt{c_1k} \norm{\mathcal{B}(X)-d}_1 + 2
\dist(X,\S_+^n) -\tr(X^\natural)\geq \frac{1}{2} \norm{X-X^\natural}_1.
\end{equation}
Note that for each $i = 1,\dots, \lfloor m/2 \rfloor$, by the construction \eqref{eq: combinedBd}, we have  
\begin{equation*}
| \ip{\cB^*(e_i),X}-d_i|\leq \frac{1}{2}|\ip{\cA^*(e_{2i-1}),X}-b_{2i-1}|+ \frac{1}{2} |\ip{\cA^*(e_{2i}),X}-b_{2i}|.
\end{equation*}
Combining this fact with \eqref{eq: sharpCovII}, we see that the function 
$\tr(X) + \frac{1}{2}\sqrt{c_1k} \norm{\mathcal{A}(X)-b}_1 + 
2\dist(X,\S_+^n)$ is $\frac{1}{2}$ sharp around $X^\natural$ as well. 

To prove the Lipschitz constant is bounded, we utilize \cite[Lemma 3.1]{candes2013phaselift}, which shows that w.h.p.
\begin{equation*}
    \norm{\cA(X)}\leq 1.1 \norm{X}_1.
\end{equation*}
Hence we see $\tr(X) + \frac{1}{2}\sqrt{c_1k} \norm{\mathcal{A}(X)-b}_1 + 
2\dist(X,\S_+^n)$ is $3+\sqrt{c_1 k}$ Lipschitz with respect to the $\ell_1$ norm. 
This completes the proof of \Cref{thm: sampleVSconditioning} (Covariance estimation I).

\section{Sharp problem formulations in the presence of noise}
\label{sec:stat}
In this section, we show that the sharpness of a problem \eqref{eq:abstract_problem} in the noiseless setting $b =\cA(x^\natural)$ provides (algorthmically useful) information even in the noisy setting, where $b$ is replaced by $\tilde b = \cA(x^\natural) + \delta$ with $\delta$ small or sparse.
We begin with the case where $\delta$ is small. 
\begin{proposition}
\label{prop:robust_noise}
Suppose \eqref{eq:abstract_problem} is $(\mu, r, \ell)$ sharp around $x^\natural$. Let $\delta\in W$ and set $\tilde b = b + \delta$.
\begin{itemize}
    \item If $\tilde x$ minimizes
    \begin{align*}
        \tilde F_{r,\ell}(x) \coloneqq f(x) + r\norm{\cA(x) - \tilde b} + \ell\dist(x,\cK),
    \end{align*}
    then $\norm{\tilde x - x^\natural} \leq \frac{2r}{\mu}\norm{\delta}$.
    \item If $\tilde x$ minimizes
    \begin{align*}
        \tilde F_{r,\ell}^{\textup{thresh}}(x) \coloneqq \max\left(\tilde F_{r,\ell}(x),\, F_{r,\ell}(x^\natural) + 3r\norm{\delta}\right),
    \end{align*}
    then $\norm{\tilde{x} - x^\natural}\leq \frac{4r}{\mu}\norm{\delta}$. Furthermore, $\tilde F_{r,\ell}^{\textup{thresh}}$ is $\tfrac{\mu}{2}$-sharp around its optimizers.
\end{itemize}
\end{proposition}

\begin{remark}
Suppose $F_{r,\ell}$ is $L$-Lipschitz with 
$f(x) = \norm{x}$, $\cK = V$ and $L\geq 1$ as in sparse vector recovery or low-rank matrix sensing. Then we have 
\begin{align*}
L \norm{x^\natural} \geq \abs{F_{r,\ell}(0) - F_{r,\ell}(x^\natural)} = \abs{r\norm{b}-\norm{x^\natural}} \implies (L+1)\norm{x^\natural} \geq r\norm{b}.
\end{align*}	
Hence, combining this inequality with the first item of Proposition \ref{prop:robust_noise}, we have 
\begin{align*}
	\frac{\norm{\tilde{x} - x^\natural}}{\norm{x^\natural }} \leq \frac{2L+2}{\mu} \frac{\norm{\delta}}{\norm{b}}.
\end{align*}
Thus we see that indeed the condition number of \eqref{eq:abstract_problem} controls the 
relative change of the solution to the relative perturbation to the data vector $b$:
\begin{gather*}
		\underbrace{\text{Relative change in solution}}_{\frac{\norm{\tilde{x} - x^\natural}}{\norm{x^\natural }}}
		\leq \mathcal{O}(\kappa) \cdot  \underbrace{\text{Relative change in data vector}}_{\frac{\norm{\delta}}{\norm{b}}}.\qedhere
\end{gather*}
This is analgous to the use of condition number in linear equations. 
\end{remark}

\begin{remark}
    \label{rem:algorithmically_sparse_tildeF}
When $\delta$ is nonzero, the penalization formulation $\tilde{F}_{r,\ell}$ is not necessarily a sharp function. However, the 
above proposition asserts that $\tilde F_{r,\ell}^\textup{thresh}(x)$ is still sharp. Hence, we may hope to apply the methods described in \Cref{sec:algs}, which apply to sharp functions:
On the surface, evaluating the function $\tilde F_{r,\ell}^\textup{thresh}(x)$ (and its subgradients) requires access to both $x^\natural$ and $\norm{\delta}$. 
While we will not have access to these quantities in practice, 
that is of little consequence if we only plan to apply first-order methods (as we suggest in \Cref{sec:algs}). Indeed, any first-order method applied to $\tilde F_{r,\ell}$ will behave equivalently to the first-order method applied to $\tilde F_{r,\ell}^\textup{thresh}$ until an iterate $\tilde x$ satisfying $\norm{\tilde x - x^\natural} \leq \frac{4r}{\mu}\norm{\delta}$ is found. In particular, the algorithms presented in \Cref{sec:algs} applied to $\tilde F_{r,\ell}$ will converge linearly to such a point with a rate depending on $\frac{\mu}{2}$. We emphasize that such a procedure is \emph{adaptive} (to the noise level $\norm{\delta}$) in both $\norm{\tilde x - x^\natural}$ and the rate of convergence to $\tilde x$.
\end{remark}

\begin{proof}[Proof of \cref{prop:robust_noise}]
    For notational convenience, in this proof we will drop all subscripts $r,\ell$.

    By the triangle inequality, for all $x\in V$, we have
    \begin{align*}
        \abs{F(x) - \tilde F(x)} &= \abs{r\norm{\cA(x) - b} - r\norm{\cA(x) - \tilde b}}\leq r\norm{\delta}.
    \end{align*}
    
    For the first claim, note that by optimality of $\tilde x$ in $\tilde F$, we have that
    \begin{align*}
        F(\tilde x) &\leq \tilde F(\tilde x) + r\norm{\delta}\leq \tilde F(x^\natural) + r\norm{\delta}\leq F(x^\natural) + 2r\norm{\delta}.
    \end{align*}
    Combining this inequality with $\mu$-sharpness of $F$ around $x^\natural$ proves the first claim.

    Consider the second claim. Note that $\tilde F(x^\natural) \leq F(x^\natural) + r\norm{\delta}$ so that $\tilde F$ achieves values bounded above by the threshold value of $F(x^\natural) + 3r\norm{\delta}$. Thus, the set of minimizers of $\tilde F^{\textup{thresh}}(x)$ is given by $\cX \coloneqq \set{x\in V:\, \tilde F(x) \leq F(x^\natural) + 3r\norm{\delta}}$.
    Then, if $\tilde x$ minimizes $\tilde F^{\textup{thresh}}(x)$, we must have
    \begin{align*}
        F(\tilde x) &\leq \tilde F(\tilde x) + r\norm{\delta}\leq F(x^\natural) + 4r\norm{\delta}.
    \end{align*}
    Combining this inequality with $\mu$-sharpness of $F$ around $x^\natural$ shows that $\norm{\tilde x - x^\natural} \leq \frac{4r}{\mu}\norm{\delta}$.

    It remains to show that $\tilde F^{\textup{thresh}}(x)$ is $\mu/2$ sharp around its optimizers $\cX$. By the definition of sharpness (see \cref{def:sharp_fn}), the goal is to show that for any $\bar x\in V\setminus \cX$, there exists $\tilde x\in \cX$ satisfying
    \begin{align*}
        \frac{\mu}{2}\norm{\bar x - \tilde x} \leq \tilde F^\textup{thresh}(\bar x) - \tilde F^\textup{thresh}(\tilde x).
    \end{align*}
    Note that for any $\bar x\in V\setminus\cX$ and $\tilde x\in \cX$, we have $\tilde F^\textup{thresh}(\bar x) = \tilde F(\bar x)$ and $\tilde F^\textup{thresh}(\tilde x) = F(x^\natural)+3r\norm{\delta}$.
    
    Set $\tilde x = (1-\alpha)x^\natural + \alpha \bar x$ where
    \begin{align*}
       \alpha = \frac{2r\norm{\delta}}{\tilde F(\bar x) - F(x^\natural) - r\norm{\delta}}.
    \end{align*}
    As $\tilde F(\bar x) > F(x^\natural) + 3r\norm{\delta}$, we have that $\alpha$ is well-defined and $\alpha\in[0,1]$. By convexity of $\tilde F$, 
    \begin{align*}
        \tilde F(\tilde x) &\leq (1-\alpha)\tilde F(x^\natural) + \alpha \tilde F(\bar x)\\
        &\leq (1- \alpha)(F(x^\natural) + r \norm{\delta}) + \alpha \tilde F(\bar x)\\
        &= F(x^\natural) + r\norm{\delta} + \left(\frac{2r\norm{\delta}}{\tilde F(\bar x) - F(x^\natural) -r\norm{\delta}}\right)
        \left(\tilde F(\bar x) - F(x^\natural) - r\norm{\delta}\right)\\
        &= F(x^\natural) +3r\norm{\delta}.
    \end{align*}
    We deduce that $\tilde x\in\cX$.

    Finally, we have
    \begin{align*}
        \frac{\mu}{2}\norm{\bar x - \tilde x} &= (1-\alpha)\frac{\mu}{2}\norm{x^\natural - \bar x}\\
        &\leq \frac{1-\alpha}{2}\left(F(\bar x) - F(x^\natural)\right)\\
        &\leq \frac{1}{2}\left(\frac{\tilde F(\bar x) - F(x^\natural) -3r\norm{\delta}}{\tilde F(\bar x) - F(x^\natural) - r\norm{\delta}}\right)\left(\tilde F(\bar x)-F(x^\natural)+ r\norm{\delta}\right)\\
        &= \frac{1}{2}\left(\tilde F(\bar x) - F(x^\natural) -3r\norm{\delta}\right)\left(\frac{\tilde F(\bar x) - F(x^\natural) - 3r\norm{\delta} + 4r\norm{\delta}}{\tilde F(\bar x) - F(x^\natural) - 3r\norm{\delta} + 2r\norm{\delta}}\right)\\
        &\leq \tilde F(\bar x) - F(x^\natural) -3r\norm{\delta}.
    \end{align*}
    Here, the second line follows by $\mu$-sharpness of $F$, the third line follows by the definition of $\alpha$, and the final line follows from the premise that $\tilde F(\bar x) > F(x^\natural) + 3r\norm{\delta}$.
\end{proof}

The next proposition shows that if $W$ is equipped with the $\ell_1$ norm and \eqref{eq:abstract_problem} is sharp in the noiseless case, then exact recovery continues to be possible (with linearly convergent algorithms) in the presence of grossly-but-sparsely-corrupted observations.

\begin{proposition}
\label{prop:robust_sparse}
Suppose \eqref{eq:abstract_problem} with $b = \cA(x^\natural)$ is $(\mu,r,\ell)$ sharp around $x^\natural$. Let $\cU$ be a normed Euclidean space. Let $\cB:V\to\cU$ be a linear operator with $r\norm{\cB} < \mu$ and $\delta\in \cU$. 
Here, $\norm{\cB}$ is the norm of $\cB:V\to\cU$ induced by the norms on $V$ and $\cU$.
Then, the function
\begin{align}
    \label{eq:sparse_noise}
    F_{r,\ell}(x) + r\norm{\cB(x) - \cB(x^\natural) - \delta}
\end{align}
is $(\mu - r\norm{\cB})$-sharp around $x^\natural$.
\end{proposition}
\begin{proof}
Let $\tilde F(x) \coloneqq F_{r,\ell}(x) + r \norm{\cB(x) - \cB(x^\natural) - \delta}$.
Let $x\in V$. Then,
\begin{align*}
    \tilde F(x) - \tilde F(x^\natural) &= F_{r,\ell}(x) - F_{r,\ell}(x^\natural) + r\norm{\cB(x) - \cB(x^\natural) -\delta} - r \norm{\delta}\\
    &\geq \mu\norm{x - x^\natural} -r \norm{\cB(x) - \cB(x^\natural)}\\
    &\geq \left(\mu - r \norm{\cB}\right)\norm{x - x^\natural}.
\end{align*}
Here, the first inequality is the triangle inequality and the second inequality is the definition of the operator norm.
\end{proof}

\begin{example}
\label{rem:sparse_noise}
Consider the phase retrieval problem where an $\alpha$-fraction of the observations are corrupted arbitrarily. Formally, consider the following procedure:
Let $V = (\S^n,\norm{\cdot}_1)$ and $\tilde W = (\R^{\tilde m},\norm{\cdot}_1)$.
Let $m = \ceil{(1-\alpha)\tilde m}$.
Fix $X^\natural \in\S^n_+\subseteq V$ to be rank-one and let $\tilde \cA:V\to\tilde W$ be the random linear map
\begin{align*}
    \tilde\cA(X)_i = g_i^\intercal X g_i,\quad \text{where}\quad g_i\sim N(0,I/\tilde m).
\end{align*}
Let $\delta\in\tilde W$ denote an arbitrary vector, chosen possibly adversarially, with only the guarantee that $\supp(\delta)\subseteq[m+1,\tilde m]$. 
Set $\tilde b = \tilde\cA(X^\natural) + \delta$. Our goal is to recover $X^\natural$ from $\tilde\cA$ and $\tilde b$.

Let $W,\, U$ denote the decomposition of $\tilde W$ along the coordinates $[m]$ and $[m+1,\tilde m]$ with the $\ell_1$ norms.
Let $\cA: V\to W$ and $\cB:V\to U$ denote the corresponding restrictions of $\tilde \cA$.
Slightly abusing notation, we will let $\delta\in U$ denote the restriction of $\delta\in\tilde W$.

Now, suppose that
\begin{align*}
    \min_{X\in\S^n_+}\set{\tr(X):\, \begin{array}{l}
    \cA(X) = \cA(X^\natural)\\
    X\in\S^n_+
    \end{array}}
\end{align*}
is $(\mu,r,\ell)$ sharp around $X^\natural$. As $\tilde W$ carries the $\ell_1$ norm and $W,\,U$ are coordinate subspaces, we may write
\begin{align}
    \label{eq:exact_penalty_sparse_noise}
    &\tr(X) + r\norm{\tilde\cA(X) - \tilde b} + \ell \dist(X,\S^n_+)\\
    &\qquad = \tr(X) + r\norm{\cA(X) - \cA(X^\natural)} + r\norm{\cB(X) - \cB(x^\natural) - \delta} + \ell \dist(X,\S^n_+).\nonumber
\end{align}
\cref{prop:robust_sparse} states that if $\mu > r\norm{\cB}$, then $X^\natural$ is the unique minimizer of \eqref{eq:exact_penalty_sparse_noise} despite the corruption $\delta$.
By \cite[Corollary 5.35]{vershynin_2012}, we know that w.h.p., $\norm{\cB} \leq \left(\tfrac{\sqrt{\alpha\tilde m}+2\sqrt{n}}{\sqrt{\tilde m}}\right)^2\leq 2\alpha + 8(n/\tilde m)$.
Combining these bounds, we have that $X^\natural$ is the unique minimizer of the sharp unconstrained minimization problem \eqref{eq:exact_penalty_sparse_noise} if
$\alpha \lesssim \mu/r$ and $n/\tilde m\lesssim\mu/r$. In particular, $\mu/r$ controls the fraction of allowed gross corruption in the observation $\tilde \cA(X^\natural)$.
\end{example}

\section{First-order methods for non-Euclidean sharp minimization}
\label{sec:algs}

We next turn to algorithms for minimizing the sharp nonsmooth formulations proposed in this paper.
Our goal is to take advantage of the well-conditioning results (\cref{thm: sampleVSconditioning}) and to design a first-order method tailored to sharp Lipschitz functions in the $\ell_1$ norm.

Before we discuss our new algorithm, we first point out that sharp functions that are well-conditioned in $\ell_1$ are necessarily poorly conditioned in $\ell_2$.
\begin{proposition}
\label{prop:poor_ell_2}
Let $V$ be one of $\R^n$, $\S^n$ or $\R^{n\by N}$. In the latter case, assume that $n\leq N$.
Suppose $f:V\to\R$ is $\mu$-sharp and $L$-Lipschitz in the $\ell_1$ norm around $x^\natural \in V$.
If $f$ is $\alpha$-sharp and $\beta$-Lipschitz in the $\ell_2$ norm around $x^\natural \in V$, then $\alpha$ and $\beta$ must satisfy
\begin{align*}
    \alpha \leq L,\qquad\text{and}\qquad
    \beta \geq \mu\sqrt{n}.
\end{align*}
\end{proposition}
\begin{proof}
First, suppose $V=\R^n$.
Let $y = x^\natural + (\text{any 1-sparse vector})$.
Then,
\begin{align*}
    \alpha\norm{y - x^\natural}_2 \leq f(y)-f(x^\natural) \leq L\norm{y-x^\natural}_1.
\end{align*}
We deduce that $\alpha \leq L$.
Similarly, let $z = x^\natural + \mb 1_n$ (the all-ones vector).
Then,
\begin{align*}
    \beta \norm{z - x^\natural}_2 \geq f(z) - f(x^\natural) \geq \mu\norm{z - x^\natural}_1.
\end{align*}
We deduce that $\beta \geq \mu\sqrt{n}$.

The other two settings are proved analogously: take $y = x^\natural + (\text{any rank-1 matrix})$ and $z = x^\natural + I_n$ where $I_n$ is the $n\by n$ identity matrix padded to $n\by N$.
\end{proof}
This proposition tells us that the functions $F_{r,\ell}$ from \cref{thm: sampleVSconditioning} have constant sharpness in the $\ell_2$ norm and an $\Omega(\sqrt{n})$ Lipschitz constant in the $\ell_2$ norm. Thus, first-order methods for sharp Lipschitz functions in the $\ell_2$ norm~\cite{goffin1977convergence,yang2018rsg,polyak1969minimization} can at best guarantee a convergence rate of $O\left(n\log\left(\frac{1}{\epsilon}\right)\right)$.
This dependence on $n$ precludes the use of such algorithms (for example, subgradient descent with Polyak stepsizes~\cite{polyak1969minimization}) on large-scale instances of our signal recovery problems (see \cref{subfig:polyak_rgd}) and motivates our development of an algorithm with almost dimension-free iteration complexity guarantees in the next subsection.

\begin{figure}
    \centering
    \begin{subfigure}[t]{0.45\linewidth}
        \centering
        \includegraphics[width=\linewidth]{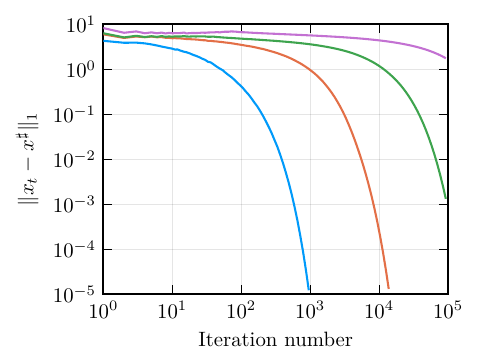}
        \caption{Subgrad.\ desc.\ with Polyak stepsizes}
        \label{subfig:polyak_rgd}
    \end{subfigure}\qquad
    \begin{subfigure}[t]{0.45\linewidth}
        \centering
        \includegraphics[width=\linewidth]{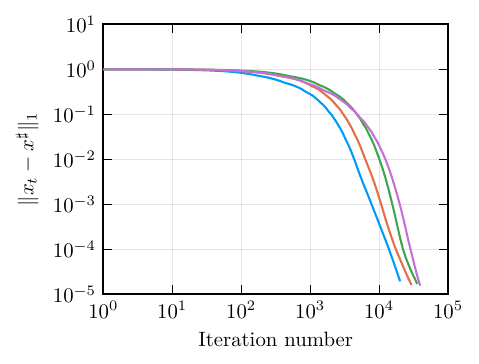}
        \caption{\texttt{Polyak-RMD}}
    \end{subfigure}
    \includegraphics[width=0.55\linewidth]{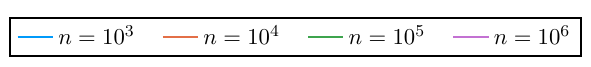}
    \caption{The convergence of subgradient descent with Polyak stepsizes~\cite{polyak1969minimization} and \texttt{Polyak-RMD} (described in Section \ref{subsec:rmd}) on sparse recovery for different values of $n$, the ambient dimension. The support size of $x^\sharp$ is $k = 5$ 
    and the sensing map $\cA$ is generated as described in Section \ref{sec: sensing_model_sample_complexity}. The number of observations is $m = 2T$ where $T = \left(2k\log(n/k) + 1.25 k + 1\right)$ is the current best estimate of the statistical threshold for sparse recovery~\cite{chandrasekaran2012convex}.  See \texttt{\url{https://github.com/alexlihengwang/sharpness_well_conditioning}} for implementation details. Note 
    subgradient descent with Polyak stepsizes deteriorates as the dimension $n$ increases while our \texttt{Polyak-RMD} is insensitive to the dimension in terms of iteration number.}

    \label{fig:numerical}
\end{figure}

\subsection{Restarted Mirror Descent}
\label{subsec:rmd}
We now describe the restarted mirror descent (RMD) algorithm, \cref{alg:rmd}.
This algorithm generalizes similar algorithms for minimizing sharp functions in a Euclidean norm~\cite{goffin1977convergence,yang2018rsg,polyak1969minimization} and has nearly dimension-independent linear convergence rates that depend explicitly on sharpness (see \cref{thm:sharp_rmd}) in an $\ell_p$ norm ($p\in[1,2]$). \cref{alg:rmd} can be applied to $F_{r,\ell}$, the sharp exact penalty formulation of \eqref{eq:abstract_problem}, or any of its sharp perturbations (see \cref{prop:robust_sparse,prop:robust_noise}).

We restrict our attention to sharp Lipschitz convex functions in an $\ell_p$ or Schatten $p$-norm for $p\in[1,2]$: Throughout this section, 
let $V$ be a normed Euclidean space.
We will overload notation so that we can simultaneously consider three separate settings:
\begin{assumption}
    \label{as:lp_setup}
Let $p\in[1,2]$ and either:
\begin{itemize}
    \item Let $V = (\R^n,\norm{\cdot}_p)$ where $\norm{\cdot}_p$ is the $\ell_p$ norm. Or,
    \item let $V = (\S^n,\norm{\cdot}_p)$ where $\norm{\cdot}_p$ is the Schatten $p$-norm. Or,
    \item let $V = (\R^{n\by N}, \norm{\cdot}_p)$ where $\norm{\cdot}_p$ is the Schatten $p$-norm.
\end{itemize}
If we are in the third case, we will assume that $n\leq N$.
\end{assumption}

Recall the mirror descent algorithm and its guarantee~\cite[Theorem 4.2]{bubeck2015convex}.

\begin{algorithm}
    \caption{Mirror Descent}
    Given $f:V\to\R$, $\bar x\in\R^d$, $\eta>0$, $T\in\N$, $h:V\to\R$
    \begin{itemize}
        \item Let $x_0 = \bar x$,\, $\theta_0 = 0\in\R^d$
        \item For $t = 1,\dots, T$
        \begin{itemize}
            \item Let $g_t\in \partial f(x_{t-1})$ and set $\theta_t = \theta_{t-1} - \eta \cdot g_t$
            \item Set $x_t = (\grad h)^{-1}(\theta_t)$
        \end{itemize}
        \item Output the $x_t$ minimizing $f(x_t)$ among $t\in[0,T]$.
    \end{itemize}
\end{algorithm}

\begin{lemma}
    \label{lem:mirr_desc_guarantee}
    Let $h:V\to\R$ be differentiable and $\sigma$-strongly convex with respect to the norm on $V$. Suppose $f:V\to\R$ is convex and $L$-Lipschitz with respect to the norm on $V$. Then, the mirror descent algorithm initialized at $\bar x$, run for $t$ iterations, with step-size $\eta$, produces $\tilde x$ such that 
\begin{align*}
    f(\tilde x) - f(x^*) \leq \frac{L^2\eta}{2\sigma} + \frac{D_h(x^* || \bar x)}{\eta t}
\end{align*}
where $x^*\in V$ is any minimizer of $f$. Here, each iteration requires computing a single subgradient of $f$ and applying $(\grad h)^{-1}$ and arithmetic operations in $V$ and $V^*$.
\end{lemma}

In the bound above, the quantity $D_h(\cdot||\cdot)$ is a Bregman divergence term. We elaborate on this term for our choice of the map $h$.
Let $p\in(1,2]$ and define $h_{\bar x}(x)\coloneqq \frac{1}{2}\norm{x - \bar x}_p^2$ for $\bar x\in V$. 
It is known~\cite[Theorem 1]{ball2002sharp} that in each of the three setups in \cref{as:lp_setup}, that $h_{\bar x}$ is differentiable and $(p-1)$-strongly convex w.r.t.\ $\norm{\cdot}_p$.
Furthermore, for all $x^*\in V$, the Bregman divergence (associated to $h_{\bar x}$) of $x^*$ with respect to $\bar x$ is
\begin{align*}
    D_{h_{\bar x}}(x^* || \bar x) \coloneqq h_{\bar x}(x^*) - h_{\bar x}(\bar x) - \ip{\grad h_{\bar x}(\bar x), x^* - \bar x} = h_{\bar x}(x^*) = \frac{1}{2}\norm{x^* - \bar x}_p^2.
\end{align*}
Thus, if $f$ is $L$-Lipschitz w.r.t.\ $\norm{\cdot}_p$, then the output $\tilde x= \texttt{mirror}(h_{\bar x}, f, L, \bar x, t, \eta)$ has suboptimality bounded by
\begin{align*}
    f(\tilde x) - f(x^*) \leq \frac{L^2\eta}{2(p-1)}+\frac{\norm{x^* - \bar x}_p^2}{2\eta T}.
\end{align*}
This bound holds simultaneously for all minimizers $x^*\in V$ of $f$.

\begin{remark}
Consider a setup from \cref{as:lp_setup} and let $p\in(1,2]$. In each iteration of mirror descent (applied with $h_{\bar x}$), we must evaluate $(\grad h_{\bar x})^{-1}$ on some input $\theta\in V^*$. By \cite[Corollary 23.5.1]{rockafellar1997convex}, this is equivalent to evaluating $\grad h^*_{\bar x}(\theta)$ where $h^*_{\bar x}$ is the convex conjugate of $h_{\bar x}$. Thus,
\begin{align*}
    (\grad h_{\bar x})^{-1}(\theta) = \grad h^*_{\bar x}(\theta) = \bar x + \grad \frac{1}{2}\norm{\theta}_q^2 = \bar x + \frac{\sign(\theta)\circ \abs{\theta}^{q-1}}{\norm{\theta}_q^{q-2}}.
\end{align*}
Here, $q$ is the H\"older dual to $p$ and the expression $\sign(\theta)\circ \abs{\theta}^{q-1}$ is applied entrywise to the entries of $\theta$ if $\theta$ is a vector and to the singular values of $\theta$ if $\theta$ is a matrix.
\end{remark}
\begin{remark}
In the matrix setting, computing the mirror map requires performing an SVD in each iteration. We expect that this full SVD may be replaced by a partial SVD near an optimal low-rank solution (following~\cite{garber2022efficient}) and leave this extension for future work.
\end{remark}

Now consider the following restarted variant of mirror descent where the step size \emph{and} mirror map $h$ update at each restart.
\begin{algorithm}
    \caption{$\texttt{RMD}(f, L, x_0, K, \set{\eta_k}, t,p)$}
    \label{alg:rmd}
    \begin{itemize}
        \item For $k = 1,2,\dots,K$
        \begin{itemize}
            \item Set $x_k \gets \texttt{mirror}(h_{x_{k-1}}, f, L, x_{k-1}, t, \eta_k)$ where
            \begin{align*}
                h_{x_{k-1}}(x)\coloneqq \frac{1}{2}\norm{x - x_{k-1}}_p^2
            \end{align*}
        \end{itemize}
        \item Output $x_K$
    \end{itemize}
    \end{algorithm}
    
    \begin{theorem}
    \label{thm:sharp_rmd}
    Consider a setup from \cref{as:lp_setup} and let $p\in(1,2]$. Suppose $f:V\to\R$ is $L$-Lipschitz and $\mu$-sharp w.r.t.\ $\norm{\cdot}_p$ and $x_0\in V$ satisfies $f(x_0)- f^* \leq \epsilon_0$.
    Let $\epsilon_k = \epsilon_0 e^{-k/2}$, $K = \ceil{2\ln\left(\frac{\epsilon_0}{\epsilon}\right)}$, $t = \ceil{\frac{e L^2}{\mu^2 (p-1)}}$, and $\eta_k = \frac{(p-1)\epsilon_k}{L^2}$.
    Then, each iterate $x_k$ in $\texttt{RMD}(f,L, x_0, K, \set{\eta_k}, t, p)$ satisfies
    $f(x_k) - f^* \leq \epsilon_k$.
    In particular, an $0<\epsilon\leq \epsilon_0$-optimizer can be computed in
    \begin{align*}
        O\left(\frac{L^2}{\mu^2(p-1)}\log\left(\frac{\epsilon_0}{\epsilon}\right)\right)
    \end{align*}
    total mirror descent steps.
    If $f:V\to\R$ is $L$-Lipschitz and $\mu$-sharp w.r.t.\ $\norm{\cdot}_1$, we may apply the above statement with $p= 1+\frac{1}{\ln n}$, sharpness $\mu$, and Lipschitz constant $eL$ w.r.t.\ $\norm{\cdot}_p$.
    \end{theorem}
    \begin{proof}
    The claim holds for $k = 0$.    
    Now let $k\geq 1$.
    By $\mu$-sharpness of $f$, there exists an optimizer $x^*$ of $f$, satisfying $\mu\norm{x_{k-1} - x^*}_p \leq f(x_{k-1}) - f^*$.
    By \cref{lem:mirr_desc_guarantee}, we have
    \begin{align*}
        f(x_k) - f^* &\leq \frac{L^2 \eta_k}{2(p-1)} + \frac{\norm{x^* - x_{k-1}}_p^2}{2\eta_k t}\\
        &\leq \frac{L^2 \eta_k}{2(p-1)} + \frac{e\epsilon_k^2}{2\mu^2\eta_k t}\\
        &\leq \frac{L^2 \eta_k}{2(p-1)} + \frac{\epsilon_k^2 (p-1)}{2\eta_k L^2}\\
        &= \epsilon_k.
    \end{align*}

    The setting of $L$-Lipschitz, $\mu$-sharp convex functions w.r.t.\ $\norm{\cdot}_1$ reduces to the setting of $p = 1+ \frac{1}{\ln(n)}$ by the bounds
    \begin{align*}
        \frac{1}{e}\norm{w}_1 &\leq \norm{w}_p \leq \norm{w}_1,\qquad\forall w\in V,
    \end{align*}
    which hold~\cite{ball2002sharp} in each of the setups in \cref{as:lp_setup}. Specifically, these inequalities imply that $f$ is $\mu$-sharp and $eL$-Lipschitz w.r.t.\ $\norm{\cdot}_p$.
\end{proof}

\begin{remark}
\label{rem:roulet_comparison}
In \cite{roulet2017sharpness}, the authors consider restarted versions of various accelerated first order methods for sharp problems including problems in non-Euclidean spaces.
In the non-Euclidean setting, \cite{roulet2017sharpness} suggests algorithms for functions $f$ that satisfy the following proxy for sharpness:
    \begin{align}
        \label{eq:bregman_sharp}
        f(x) - f(x^\natural) &\geq \mu \sqrt{D_h(x^\natural|| x)}.
    \end{align}
    Here, $h:V\to\R$ is a \emph{fixed} differentiable and $\sigma$-strongly convex function\footnote{The notation of \cite{roulet2017sharpness} uses a $1$-strongly convex function, but this only changes the value of $\mu$ by a factor of $\sqrt{\sigma}$.} and $D_h$ is the Bregman divergence induced by $h$.
    Then, calculations almost identical to the proof of \cref{thm:sharp_rmd} show that restarted mirror descent with the mirror map $h$ in every restart produces an $\epsilon$-suboptimal solution in $O\left(\frac{L^2}{\mu^2\sigma}\log\left(\frac{\epsilon_0}{\epsilon}\right)\right)$ mirror descent steps.

    Unfortunately, \eqref{eq:bregman_sharp} is \emph{not} implied by sharpness in the original norm and may be overly restrictive, as illustrated in the example below. Specifically, in standard prox setups for mirror descent in $\ell_p$ norms or Schatten-$p$ norms, \eqref{eq:bregman_sharp} cannot hold for any $\mu>0$ unless $x^\natural$ has full support in the vector case or full rank in the matrix case.
\end{remark}

    \begin{example}
    Let $V=(\R^n,\norm{\cdot}_p)$ where $p\in(1,2)$ and $n\geq 2$.
    Let $f(x) \coloneqq \norm{x - x^\natural}_p$ where $x^\natural = e_1$. Clearly, $f$ is $1$-sharp around $x^\natural$ and $1$-Lipschitz w.r.t.\ $\norm{\cdot}_p$. 

    Two standard prox setups for mirror descent~\cite{nesterov2013first} in this setting are to take $h(x) = \frac{1}{2}\norm{x}_p^2$ (in the unbounded case) or
    $h(x) = \frac{1}{p}\norm{x}_p^p$ (in the bounded case); see \cite[Theorem 2.1]{nesterov2013first} for the corresponding strong convexity parameters.
    Consider the first setting, i.e., $h(x) = \frac{1}{2}\norm{x}_p^2$.    
    Letting $x_\epsilon = x^\natural + \epsilon e_2$, we have by Bernoulli's inequality that
    \begin{align*}
        \sqrt{D_h\left(x^\natural || x\right)} &= \sqrt{\frac{1}{2}(1+\abs{\epsilon}^p)^{2/p} - \frac{1}{2} - \ip{e_1, \epsilon e_2}}\\
        &=2^{-1/2} \sqrt{(1+\abs{\epsilon}^p)^{2/p} - 1}\\
        &\geq p^{-1/2} \abs{\epsilon}^{p/2}.
    \end{align*}
    On the other hand,
        $f(x_\epsilon) - f(x^\natural) \leq \abs{\epsilon}$.
    Thus, letting $\abs{\epsilon}\to 0$, we see that \eqref{eq:bregman_sharp} cannot hold for \emph{any} $\mu>0$.

    Next, suppose $h(x) = \frac{1}{p}\norm{x}^p_p$. Then,
    \begin{align*}
        \sqrt{D_h(x^\natural||x)} &= \sqrt{\frac{1}{p}(1 + \abs{\epsilon}^p) - \frac{1}{p} - \ip{e_1,\epsilon e_2}}= \abs{\epsilon}^{p/2}/\sqrt{p}.
    \end{align*}
    Again, comparing this bound to the fact that $f(x_\epsilon) - f(x^\natural) \leq \abs{\epsilon}$, we deduce that \eqref{eq:bregman_sharp} cannot hold for \emph{any} $\mu>0$.
\end{example}

The following variant of~\cref{alg:rmd}
replaces knowledge of $\mu$ with knowledge of $f^*$.
This ``Polyak'' variant achieves the same convergence rate as \Cref{alg:rmd} with the optimal choice of $\mu$.

\begin{algorithm}
    \caption{$\texttt{Polyak-RMD}(f, L, x_0, K, \set{\eta_k}, f^*,\set{\epsilon_k}, p)$}
    \label{alg:polyak-rmd}
    \begin{itemize}
        \item For $k = 1,2,\dots,K$
        \begin{itemize}
            \item Run mirror descent         
            $\texttt{mirror}(h_{x_{k-1}}, f, L, x_{k-1}, \infty, \eta_k)$ 
            until it finds an iterate $x_k$ satisfying
            \begin{align*}
                f(x_k) - f^* \leq \epsilon_k
            \end{align*}
        \end{itemize}
        \item Output $x_K$
    \end{itemize}
    \end{algorithm}

The convergence guarantee for Polyak-RMD (\cref{alg:polyak-rmd}) and its proof are entirely identical to that of \cref{alg:rmd}.

\begin{proposition}
    Consider a setup from \cref{as:lp_setup} and let $p\in(1,2]$. Suppose $f:V\to\R$ is $L$-Lipschitz and $\mu$-sharp w.r.t.\ $\norm{\cdot}_p$ and $x_0\in V$ satisfies $f(x_0)- f^* \leq \epsilon_0$.
    Let $\epsilon_k = \epsilon_0 e^{-k/2}$, $K = \ceil{2\ln\left(\frac{\epsilon_0}{\epsilon}\right)}$, and $\eta_k = \frac{(p-1)\epsilon_k}{L^2}$.
    Then, each iterate $x_k$ in $\texttt{Polyak-RMD}(f,L, x_0, K, \set{\eta_k}, f^*, \set{\epsilon_k}, p)$ is computed in at most
        $\ceil{\frac{eL^2}{\mu^2(p-1)}}$
    mirror descent steps.
    In particular, an $0<\epsilon\leq \epsilon_0$-minimizer can be computed in
    \begin{align*}
        O\left(\frac{L^2}{\mu^2(p-1)}\log\left(\frac{\epsilon_0}{\epsilon}\right)\right)
    \end{align*}
    total mirror descent steps.
    If $f:V\to\R$ is $L$-Lipschitz and $\mu$-sharp w.r.t.\ $\norm{\cdot}_1$, we may apply the above statement with $p= 1+\frac{1}{\ln n}$, sharpness $\mu$, and Lipschitz constant $eL$ w.r.t.\ $\norm{\cdot}_p$.
\end{proposition}

\begin{remark}
    \label{rem:}
    The basic algorithmic structure of \cref{alg:rmd,alg:polyak-rmd} can also be extended to an adaptive variant. Specifically, it is possible to design an adaptive version of \cref{alg:rmd} (similar in spirit to~\cite{renegar2022simple}) that does not require knowledge of either $f^*$ or $\mu$ but would have a convergence rate of $O\left(\frac{L^2}{\mu^2(p-1)}\left(\log\left(\frac{\epsilon_0}{\epsilon}\right)\right)^2\right)$.
\end{remark}

\section{Discussion}\label{sec: discussion}

This paper shows that for various statistical signal recovery tasks, once the 
sample size is greater than a constant multiple of the recovery threshold, 
then the convex optimization problem becomes well-conditioned in the sense of sharpness w.r.t.\ the $\ell_1$ or Schatten-$1$ norm. In turn, this fact shows 
the optimization problem is robust to measurement error and optimization error. Furthermore, the newly developed algorithm \texttt{RMD} is able to achieve nearly-dimension-independent convergence rates. 

We hope this paper induces interest in the interplay between statistics and convex optimization, especially in nonsmooth formulations and algorithms for these nonsmooth formulations. In particular, the 
following two directions might be of interest for future investigations:
\begin{itemize}
	\item \textbf{Well conditioning beyond RIP}: This paper considers recovery tasks where strong RIP bounds have been established. Can we prove well-conditioning results for other statistical problems such as matrix completion \cite{candes2010power} and phase retrieval with coded diffraction patterns \cite{candes2015phase} where standard RIP bounds may not hold? \item \textbf{Adaptive penalty parameters}: Given a concrete statistical model, we can give upper bounds on $r$ and $\ell$ so that 
	the corresponding function $F_{r,\ell}$ is $\mu$-sharp. On the other hand, determining these parameters may be difficult without a precise statistical model. Thus, are there algorithmic ways of estimating or adaptively choosing these two parameters?
\end{itemize}
 
\section*{Acknowledgments}
Part of this research was completed in part while Alex L.\ Wang was supported by the Dutch Scientific Council (NWO)
grant OCENW.GROOT.2019.015 (OPTIMAL). Lijun Ding was  supported by an NSF TRIPODS grant to the Institute for Foundations of Data Science (NSF DMS-2023239) and
by DOE ASCR Subcontract 8F-30039 from Argonne National Laboratory.
Lijun Ding would like to thank Dmitriy Drusvyatskiy and Mateo Diaz for helpful discussions.
{
    \bibliographystyle{abbrv}

}
\appendix

\section{Proof of Proposition \ref{prop:genSharpnessAndConditioning} for sparse vector recovery}\label{sec:  Proof of Proposition sharpSparse}
\begin{proof}
	Let $\delta\in\R^n$ be arbitrary. Our goal is to show that
	\begin{align*}
		\norm{x^\natural + \delta}_1 + \sqrt{k'} \norm{A\delta}_1 - \norm{x^\natural} \geq \frac{\epsilon}{2+\epsilon}\norm{\delta}.
	\end{align*}
	Without loss of generality, we may reindex $\R^n$ so that $\supp(x^\natural)\subseteq[k]$ and $\abs{\delta_{k+1}}\geq \abs{\delta_{k+2}} \geq \dots \geq \abs{\delta_n}$.
	We decompose $\delta$ as a sum of vectors $\delta_k + \sigma_1 + \dots + \sigma_t$, each in $\R^n$. Specifically, $\delta_k$ extracts the first $k$ coordinates of $\delta$, $\sigma_1$ extracts the next $k'$ coordinates of $\delta$, and each subsequent $\sigma_i$ extracts the next $k'$ coordinates of $\delta$. Finally, $\sigma_t$ may extract fewer than $k'$ coordinates of $\delta$.
	Let $\delta_{k^\perp} \coloneqq \delta - \delta_k = \sigma_1 + \dots + \sigma_t$.
	
	We can bound
	\begin{align*}
		\norm{\delta_{k^\perp}}_1 &\geq \sum_{i=1}^{t-1} \norm{\sigma_i}_1 && \text{($\sigma_i$ are disjoint)}\\
		&= k'\sum_{i=1}^{t -1} \frac{\norm{\sigma_i}_1}{k'}\\
		&\geq k'\sum_{i=2}^{t}\norm{\sigma_i}_\infty && (\abs{\delta_{k+1}}\geq\dots\geq\abs{\delta_n})\\
		&\geq \sqrt{k'}\sum_{i=2}^t \norm{\sigma_i}_2\\
		& \geq \sqrt{k'}\cdot\left(\ripup_{k'}\right)^{-1}\sum_{i=2}^t \norm{A \sigma_i}_1\\
		&\geq \sqrt{k'}\cdot\left(\ripup_{k'}\right)^{-1}\left(\norm{A(\delta_k + \sigma_1)}_1 - \norm{A\delta}_1\right) && \text{(triangle inequality)}\\
		&\geq \sqrt{k'} \frac{\ripdown_{k+k'}}{\ripup_{k'}}\norm{\delta_k + \sigma_1}_2 - \sqrt{k'}\norm{A\delta}_1 && \\
		&\geq \sqrt{k'} \frac{\ripdown_{k+k'}}{\ripup_{k'}}\norm{\delta_k}_2 - \sqrt{k'}\norm{A\delta}_1 && \text{($\delta_k,\,\sigma_1$ are disjoint)}\\
		&\geq \sqrt{\frac{k'}{k}}\frac{\ripdown_{k+k'}}{\ripup_{k'}}\norm{\delta_k}_1 - \sqrt{k'}\norm{A\delta}_1\\
		&\geq (1+\epsilon) \norm{\delta_k}_1 - \sqrt{k'}\norm{A\delta}_1.
	\end{align*}
	
	We are now ready to prove sharpness:
	\begin{align*}
		&\norm{x^\natural + \delta}_1 + \sqrt{k'}\norm{A\delta}_1 - \norm{x^\natural}_1\\
		&\qquad\geq \norm{\delta_{k^\perp}}_1 - \norm{\delta_k}_1 +\sqrt{k'}\norm{A\delta}_1 && \text{($x^\natural+\delta_k,\,\delta_{k^\perp}$ are disjoint)}\\
		&\qquad= \left(\frac{2}{2+\epsilon} + \frac{\epsilon}{2+\epsilon}\right)\norm{\delta_{k^\perp}}_1 - \norm{\delta_k}_1 +\sqrt{k'}\norm{A\delta}_1\\
		&\qquad\geq \left(\frac{2}{2+\epsilon}(1+\epsilon) - 1\right)\norm{\delta_k}_1 + \frac{\epsilon}{2+\epsilon}\norm{\delta_{k^\perp}}_1 && \text{(bound on $\norm{\delta_{k^\perp}}_1$)}\\
		&\qquad\qquad + \left(1 - \frac{2}{2+\epsilon}\right)\sqrt{k'}\norm{A\delta}_1\\
		&\qquad \geq \frac{\epsilon}{2+\epsilon}\norm{\delta}_1.
	\end{align*}
	This proves the claim as $\delta\in \R^n$ was arbitrary.
\end{proof} 
\end{document}